%% file: dyadic2026.tex
\documentclass[11pt,reqno]{amsart} % for arxiv-version
\usepackage{graphicx} % Required for inserting images
\usepackage{amsmath}
\usepackage{amssymb}
\usepackage{amsthm}
\usepackage{hyperref}
\usepackage{geometry}
\usepackage[all]{xy}
\usepackage{tikz}
\usetikzlibrary{decorations.pathreplacing, arrows.meta}

\newtheorem{thm}{Theorem}[section]
\newtheorem{prop}[thm]{Proposition}
\newtheorem{lemma}[thm]{Lemma}
\newtheorem{defi}[thm]{Definition}
\newtheorem{example}[thm]{Example}

\newtheorem{cor}[thm]{Corollary}

\newcommand{\PP}{\mathbb{P}}
\newcommand{\EE}{\mathbb{E}}

\newcommand{\RR}{\mathbb{R}}

\newcommand{\len}{\mathord{\ell}}

\title{Longest increasing subsequences of dyadic-type chaotic orbits}
\author[Shinsuke Iwao]{Shinsuke Iwao}
\address{Faculty of Business and Commerce, Keio University, Kanagawa 223-8521, Japan}
\email{iwao-s@keio.jp}

\author[Fumihiko Nakamura]{Fumihiko Nakamura}
\address{Faculty of Engineering, Kitami Institute of Technology, Hokkaido, 090-8507, JAPAN}
\email{nfumihiko@mail.kitami-it.ac.jp}

\author[Yushi Nakano]{Yushi Nakano}
\address{Faculty of Science, Hokkaido University, Hokkaido, 060-0810, Japan}
\email{yushi.nakano@math.sci.hokudai.ac.jp}


\begin{document}
\begin{abstract}
This paper studies the longest increasing subsequence (LIS) problem for sequences generated by dyadic-type chaotic interval maps.
Starting from a single point \(x\in[0,1)\) chosen uniformly at random, we form the order pattern of the first \(N\) points of its orbit, with the doubling map as the basic model.
Let \(\lambda_1^{(N)}\) be the LIS length, equivalently the length of the first row of the Young diagram obtained by Schensted's insertion.
We show that \(\mathbb E[\lambda_1^{(N)}]/\sqrt N\to 2\), matching the leading asymptotics in the classical Ulam--Hammersley problem for uniform random permutations.
\end{abstract}

\maketitle

\section{Introduction}
%\marginpar{\tiny YN: そもそも論文内でLISについてしか議論していないので、また歴史としてはUlam--Hammersleyの方がlimit shapeより古いので（当たり前ですが）、{\bf タイトルを変えました}。一方でyoungタブローの語を出した方がreaderが増えそうというメリットを消したくないので、可能な限り早い段階でyoungタブローの言葉をイントロに入れました}
This paper concerns the longest increasing subsequence (LIS) problem for sequences generated by dyadic-type chaotic interval maps.
In its classical form, the LIS problem asks how long an increasing subsequence one can find in a finite sequence.
For independent random variables with a continuous common distribution, the induced order pattern is a uniform random permutation; this is the classical Ulam--Hammersley problem.
Through Schensted's correspondence, the length of the LIS is also the length of the first row of the Young diagram associated with the permutation~\cite{Schensted1961}.
Thus the LIS problem is one of the most basic entry points to the asymptotic theory of random Young diagrams.
The Logan--Shepp and Vershik--Kerov theorems for the limit shape imply the first-order asymptotic $2\sqrt{N}$ for the LIS, and the Baik--Deift--Johansson theorem identifies the $N^{1/6}$-scale fluctuations with the Tracy--Widom distribution~\cite{LoganShepp1977,VershikKerov1977,BaikDeiftJohansson1999}.

Here we study an analogue of the Ulam--Hammersley problem generated by a deterministic orbit.
We choose a single initial point $x\in[0,1)$ uniformly at random and follow its orbit under a  map $f:[0,1)\to [0,1)$ in a dyadic symbolic class.
The basic model is the doubling map
%\marginpar{\tiny SI: $f(x)=2x\bmod 1$?}
\[
    f(x)=2x\bmod 1.
\]
The same symbolic viewpoint also covers the tent map, namely the piecewise linear map with slope $2$ on the left half of the interval and slope $-2$ on the right half.
%\marginpar{\tiny YN: the full quadratic map $Q(x)=4x(1-x)$についてのcorもどこかで言及: ただし、期待値の選択はLebではない: しかし、the aceipのdensityは一応Leb-integrableなので、（$L^\infty$-density away from 0だと上下で挟んですぐにLebに帰着できそうなので）$[N^{-1},N]$-cut off $h_N$でいい感じにLebに帰着できないか？}
%\marginpar{\tiny Furthermore, through the standard topological conjugacy with the tent map, the same conclusion also holds for the logistic map $Q(x)=4x(1-x)$ with respect to its absolutely continuous invariant probability measure. おそらく、Lebで$Q(x)=4x(1-x)$もいけそうで、スペースが欲しいのでremarkを書くと思います}
%\marginpar{\tiny YN: full branch quadratic mapの言及も入れる!}
We rank the first $N$ orbit points
$
    x,\; f(x),\; f^2(x),\ldots, f^{N-1}(x)$
and consider the LIS of the resulting order pattern.
Equivalently, by Schensted insertion, we study the first row of the Young diagram generated by this orbit.
The resulting sequence is not i.i.d., and its induced order pattern is not a uniform random permutation.
Thus, the problem does not fall under either the classical i.i.d.\ setting or the Plancherel framework.
Our main result shows that the expected LIS nevertheless has the same first-order asymptotics as in the Ulam--Hammersley problem:
if 
%$\lambda_1^{(N)}=
$\lambda_1^{(N)}(x)$ denotes the LIS of the first $N$ orbit points $\{f^n(x)\}_{n=0}^{N-1}$, then
\[
    \lim_{N\to\infty}
    \frac{\EE[\lambda_1^{(N)}]}{\sqrt{N}}
    =2.
\]
Here and throughout the paper, 
%unless otherwise stated, 
probabilities
and expectations of random variables determined by the initial point
\(x\) are taken with respect to Lebesgue measure on \([0,1)\).
%In particular,
%%$\EE[\cdot]$ is the integration with respect to the Lebesgue measure.
%$\EE[\lambda_1^{(N)}]=\int _0^1\lambda_1^{(N)}(x)\, dx$.

A natural first guess is that such a result should follow from the strong statistical properties of expanding maps.
For many expanding or hyperbolic systems, exponential decay of correlations, often combined with martingale or reverse-martingale decompositions of Gordin--Liverani type, provides a standard route to central limit theorems and related limit laws~\cite{Gordin1969,Liverani1996,Young1998,Young1999}.
However, this mechanism does not seem to control the LIS in any direct way.
The LIS is neither an additive observable nor a local statistic depending on a fixed number of orbit coordinates.
A maximal increasing subsequence may choose many possible sets of times, and even the comparison of two orbit points can require long, overlapping tails of the binary expansion of the same initial point.
Thus an i.i.d.-type approximation obtained only from exponential mixing is not by itself suited to the non-local optimization problem considered here.

The main technical idea is to use the dyadic symbolic structure more deeply than what is captured by abstract mixing estimates.
In the symbolic coordinate used in the proof, an initial point is written as
\[
    x=0.b_1b_2b_3\cdots,
\]
and the dynamics is the left shift; under the uniform initial distribution, the digits $b_i$ are independent Bernoulli variables.
The difficulty is that comparisons of $f^i(x)$ and $f^j(x)$ involve overlapping tails of this same binary sequence.
For the lower bound, we cut the binary expansion into blocks, search for prescribed finite words that force selected orbit points into specified dyadic intervals, and use the remaining tail digits as independent labels.
After collision events are excluded, the relative order of the selected orbit points agrees in distribution with that of an independent uniform sequence; together with the order-isomorphism statement for i.i.d.~samples, this reduces the selected subsequence to the uniform permutation model.
At the lower-bound level, the precise reduction is the combination of Proposition~\ref{prop:LIS_succ} with Proposition~\ref{prop:near_unif}, followed by the uniform-permutation estimate in Theorem~\ref{thm:known}; see Section~\ref{sec:lower_uniform_comparison}.
The upper-bound version appears in Section~\ref{sec:ev_small_block}, where the comparison is applied after thinning and leads to the estimate~\eqref{eq:ev_L_2}.

%The lower and upper bounds implement this comparison in complementary ways.
%\marginpar{\tiny YN: この段落はカットでもいい}
%For the lower bound, word-search and collision avoidance produce independent uniform order patterns inside suitable time blocks, and the resulting LIS contributions are concatenated using a variational decomposition.
%For the upper bound, short periodic binary words create clusters of nearby hits with overlapping tails; we therefore use a variable-length partition into regular and exceptional cylinders, thin the sequence by removing hits that are too close, and control the removed points by concentration estimates for Bernoulli sums.
%A union bound over all cylinders and time intervals gives the required high-probability estimate, which is converted into the expectation bound.
%%\marginpar{\tiny YN: いい塩梅にこの論文と興味の対象が似ていて、しかし安全に共通部分が小さい論文群を見つけたので、それを短く言及するかもしれません。ここに書かないにしてもお二人には後で論文名を共有します}

For a different perspective based on the full distributions of order patterns generated by interval maps, see Abrams et al.~\cite{AbramsEtAl2013}.

The paper is organized as follows.
Section~2 states the main results, introduces the pattern construction
and the class of admissible orders.
%, and explains how the tent map fits into the same framework via Gray coding.
Sections~3 and~4 prove the lower and upper bounds, respectively.
%The appendix records the concentration inequality used in the proof.
%\marginpar{\tiny YN: Appendixを復活させないならここの記述を変更}

\section{The problem}

\subsection{Setting}

Let \( N \) be a positive integer, and let \( S_N \) denote the symmetric group.  
For a sequence \( a = (a_1, a_2, \dots, a_N)\in \RR^N \), we define \( \mathrm{Pat}(a) \) to be the sequence obtained by replacing the entries $a_1,a_2,\dots,a_N$ with \( 1,2,\dots,N \) in increasing order. (The smallest entry of \( a \) is replaced by \( 1 \), the next smallest by \( 2 \), and so on.)
In the case where some entries of \( a_1,\dots,a_N \) are equal, the one appearing earlier is regarded as smaller.
The sequence \( \mathrm{Pat}(a) \) can be regarded as an element of $S_N$ by identifying it with the one-line notation of a permutation.

Let \( x \in [0,1) \) be a random variable distributed uniformly.
Set $a_1:=x$ and define the sequence \( a_1, a_2, a_3, \dots \in [0,1) \) by the recursive equation $a_{n+1}=2a_n\bmod{1}$.
For the first \( N \) terms of this sequence \( a^{(N)} = (a_1, \dots, a_N) \), we set
\begin{equation}\label{eq:def_of_pi}
\pi^{(N)} := \mathrm{Pat}(a^{(N)}).
\end{equation}
Let \( \lambda^{(N)} \) be the Young diagram obtained by applying Schensted's row-insertion algorithm (see, for example \cite{Hulton}) to \( \pi^{(N)} \). 
Denoting by \( \lambda_i^{(N)} \) the number of boxes contained in the $i$-th row of \( \lambda^{(N)} \),
each \( \lambda_i^{(N)} \) becomes a random variable depending on the value of \( x \).

It is known \cite[\S 3]{Hulton} that \( \lambda_1^{(N)} \) is equal to the length of a longest increasing subsequence contained in \( a_1, a_2, \dots, a_N \).
%In what follows, we use LIS as an abbreviation for ``longest increasing subsequence''.

\subsection{Main Theorem}

We write the binary expansion of an element in \( [0,1) \) as
\[
0.b_1 b_2 b_3 \ldots \qquad (b_i \in \{0,1\}).
\]
By excluding binary sequences whose tails consist entirely of 1's, the binary expansion becomes unique.

Let $f : [0,1) \to [0,1)$ be the \textit{doubling map} given by
\begin{equation}\label{eq:double_map}
f(x)=2x \bmod{1}.
\end{equation}
In terms of binary expansions, the map \( f \) corresponds to the \textit{shift map}
\begin{equation}\label{eq:def_shift}
0.b_n b_{n+1} b_{n+2} \dots \mapsto 0.b_{n+1} b_{n+2} b_{n+3} \dots.
\end{equation}

We can extend the situation to more general settings.
Let \(\succ\) be a total order on \([0,1)\) satisfying the following property:
\begin{enumerate}
\item[($\ast$)]
For sequences \(c_1,c_2,c_3,\dots\) and \(d_1,d_2,d_3,\dots\), if there exists \(N\) such that
\(c_i = d_i\) for any \(i < N\) and \(c_N \neq d_N\), we have
\[
0.c_1 c_2 \dots c_N \succ 0.d_1 d_2 \dots d_N
\;\;\Longleftrightarrow\;\;
0.c_1 c_2 c_3 \ldots \succ 0.d_1 d_2 d_3 \dots.
\]
In other words, the order between the two numbers is determined only by their first $N$ digits.
\end{enumerate}

Let $\Omega:=\{0,1\}^{\mathbb N}$ be the set of $01$-sequences equipped with the Bernoulli product measure $\nu:=\left(\frac12,\frac12\right)^{\otimes\mathbb N}$. 
Consider the \textit{binary coding map}
\[
\widetilde{\beta}:\Omega\to [0,1],\qquad 
\widetilde{\beta}(c):=\sum_{j=1}^{\infty}c_j2^{-j}
\]
and the exceptional set
\[
\mathfrak z
:=
\left\{
(c_1,c_2,\dots)\in\Omega:
\text{\(c_n=1\) for all sufficiently large \(n\)}
\right\}.
\]
Restricting $\widetilde{\beta}$ to $\Omega\setminus
\mathfrak z$, we obtain the bijection $\beta:\Omega\setminus\mathfrak z\longrightarrow[0,1)$.
Since \(\mathfrak z\) is countable, $\beta$ is a measure-preserving map.

For a finite binary word \(w\), let
\[
[w]
:=
\{c\in\Omega:\text{\(c\) begins with \(w\)}\}
\]
be the \textit{cylinder set} corresponding to $w$.
Its image 
\begin{equation}\label{eq:def_of_cyl}
C_w:=
\beta([w]\setminus\mathfrak z)=
\{x\in[0,1):\text{the binary expansion of \(x\) begins with \(w\)}\}
\end{equation}
under $\beta$ is also referred to as a \textit{cylinder set}.

For subsets \(A,B\subset[0,1)\), write \(A\prec B\) if
\(x\prec y\) for every \(x\in A\) and \(y\in B\).
Then, the property \((\ast)\) can be rephrased as follows:
\begin{equation}\label{eq:property_one}
\begin{aligned}
&\text{For any two distinct finite words \(w\) and \(w'\) of the same lengths},\\
&\text{exactly one of
$C_{w}\prec C_{w'}$ and $C_{w'}\prec C_{w}$
holds}.
% \\ 
% &\text{where $ua$ $(a\in \{0,1\})$ is the concatenation of the word $u$ and the letter $a$.}  
\end{aligned}
\end{equation}

For a finite word $w$, an \textit{extension of $w$} is a word that begins with $w$.
\begin{lemma}\label{lemma:cylinder_property}
Let $\succ$ be a total order on $[0,1)$ satisfying \((\ast)\).
Then, for any finite words $w$ and $w'$, the following (i)--(ii) hold:
\begin{itemize}
\item[(i)] If $C_w\cap C_{w'}\neq \emptyset$, one of $w$ and $w'$ is an extension of the other.
\item[(ii)] If $C_w\cap C_{w'}=\emptyset$, exactly one of $C_{w}\prec C_{w'}$ and $C_{w'}\prec C_w$ holds.  
\end{itemize}
\end{lemma}
\begin{proof}
Let $\ell$ and $\ell'$ be the length of $w$ and $w'$, respectively. 
We may assume $\ell\leq \ell'$ without loss of generality.
Let $w''$ be the first $\ell$ digits of $w'$.
Then, we have $C_{w'}\subset C_{w''}$ by the definition of cylinder sets.

(i) If $C_w\cap C_{w'}\neq \emptyset$, then we have $C_w\cap C_{w''}\neq \emptyset$.
From \eqref{eq:property_one}, it follows that $w=w''$, hence $w'$ is an extension of $w$.

(ii) By assumption, $w'$ is not an extension of $w$, hence $w\neq w''$.
From \eqref{eq:property_one}, exactly one of $C_w\prec C_{w''}$ and $C_{w''}\prec C_{w}$ holds.
The claim follows from $C_{w'}\subset C_{w''}$.

\end{proof}

\begin{lemma}\label{lemma:ord_isom}
Suppose that \(\succ\) is a total order on \([0,1)\) satisfying
\((\ast)\).
Then there exist Lebesgue-null sets \(Z_1,Z_2\subset[0,1)\) and a
bijection
\[
\iota:Z_1^c\to Z_2^c
\]
such that
\[
x\succ y
\quad\Longleftrightarrow\quad
\iota(x)>\iota(y)
\qquad
(x,y\in Z_1^c),
\]
and \(\iota\) preserves the Lebesgue measure.
\end{lemma}

\begin{proof}
% Let
% \[
% \Omega:=\{0,1\}^{\mathbb N}
% \]
% be equipped with the Bernoulli product measure
% \[
% \nu:=\left(\frac12,\frac12\right)^{\otimes\mathbb N}.
% \]
% Let
% \[
% \mathfrak z
% :=
% \left\{
% (c_1,c_2,\dots)\in\Omega:
% \text{\(c_n=1\) for all sufficiently large \(n\)}
% \right\}.
% \]
% The set \(\mathfrak z\) is countable, and the binary coding map
% \[
% \beta:\Omega\setminus\mathfrak z\longrightarrow[0,1),
% \qquad
% \beta(c):=\sum_{j=1}^{\infty}c_j2^{-j},
% \]
% is a bijection.

% For a finite binary word \(u\), let
% \[
% [u]
% :=
% \{c\in\Omega:\text{\(c\) begins with \(u\)}\}
% \]
% be the corresponding cylinder, and let
% \[
% C_u:=\beta([u]\setminus\mathfrak z).
% \]

% For subsets \(A,B\subset[0,1)\), write \(A\prec B\) if
% \(x\prec y\) for every \(x\in A\) and \(y\in B\).
% By property \((\ast)\), for every finite word \(u\), exactly one of
% \[
% C_{u0}\prec C_{u1},
% \qquad
% C_{u1}\prec C_{u0}
% \]
% holds, where $ua$ $(a\in \{0,1\})$ is the concatenation of the word $u$ and the letter $a$.

For \(a\in \{0,1\}\), let \(\overline{a}\in \{0,1\}\) denote the other element of \(\{0,1\}\).
For a finite word $w$ ($w$ might be the empty word $\emptyset$) and $a\in \{0,1\}$, we define
\[
\sigma_w(a)
:=
\begin{cases}
0,
&
\text{if }C_{wa}\prec C_{w\overline{a}},
\\[1mm]
1,
&
\text{if }C_{w\overline{a}}\prec C_{wa},
\end{cases}
\]
where $wa$ is the concatenation of $w$ and $a$.
The map $\sigma_w(a)$ is well-defined by  \eqref{eq:property_one}.
Note that $\sigma_w(0)\neq \sigma_w(1)$.
% For each \(u\), the map
% $
% a\mapsto\sigma_u(a)
% $
% is a permutation of \(\{0,1\}\).

We define the map \(\tau:\Omega\to\Omega\);
\(c_1c_2\ldots\mapsto d_1d_2\dots\) by
\[
d_j
:=
\sigma_{c_1\cdots c_{j-1}}(c_j),
\qquad
(j\geq1).
\]
Since the first $\ell$ digits $d_1d_2\dots d_\ell$ are determined only by $c_1c_2\dots c_\ell$, we have a one-to-one correspondence $c_1c_2\dots c_\ell\leftrightarrow d_1d_2\dots d_\ell$ for any $\ell\geq 1$.
This implies that $\tau$ is a bijection of \(\Omega\).
Moreover, for any finite words $c=c_1c_2\dots c_\ell$ and $d=\tau(c)=d_1d_2\dots d_\ell$, we have
\[
\tau([c])=[d]\quad \text{and}\quad 
\nu([c])
=
\nu([d])
=
2^{-\ell},
\]
which implies that $\tau$ preserves the product measure of all cylinder sets.
Since the cylinder sets generate the Borel \(\sigma\)-algebra of
\(\Omega\), we conclude that \(\tau\) preserves \(\nu\).

If \(c,c'\in\Omega\setminus \mathfrak{z}\) first differ at position \(j\), their images \(d=\tau(c)\) and \(d'=\tau(c')\) also first differ at position \(j\).
By the definition of \(\sigma_w\), we have
\[
\begin{aligned}
\beta(c)\prec\beta(c')
&\iff
C_{c_1c_2\dots c_j}\prec C_{c_1c_2\dots \overline{c_j}}\\
&\iff
d_1d_2\dots d_j<_{\mathrm{lex}}d_1d_2\dots \overline{d_j}
\iff \tau(c)<_{\mathrm{lex}}\tau(c').
\end{aligned}
\]

Let
\(
A
:=
(\Omega\setminus\mathfrak z)
\cap
\tau^{-1}(\Omega\setminus\mathfrak z)
\)
be a maximal subset of $\Omega$ on which both of $\beta$ and $\beta\circ \tau$ are injective.
Set
\(
B:=\tau(A).
\)
Then, both $(\Omega\setminus\mathfrak z)\setminus A$ and $(\Omega\setminus\mathfrak z)\setminus B$ are countable.
Define
\[
Z_1:=[0,1)\setminus\beta(A),
\qquad
Z_2:=[0,1)\setminus\beta(B).
\]
Then \(Z_1\) and \(Z_2\) are countable, and hence Lebesgue-null.
Finally, define
\[
\iota
:=
\beta\circ\tau\circ
\bigl(\beta|_A\bigr)^{-1}
:
Z_1^c\to Z_2^c.
\]
This map is a bijection, and the preceding order comparison shows that
\[
x\prec y
\quad\Longleftrightarrow\quad
\iota(x)<\iota(y).
\]
Moreover, 
since $\beta$ and $\tau$ are measure-preserving, 
% for every Borel set \(E\subset Z_1^c\),
% \[
% \begin{aligned}
% \operatorname{Leb}(\iota(E))
% &=
% \nu\left(
% \tau\left(
% \beta^{-1}(E)
% \right)
% \right)
% \\
% &=
% \nu\left(
% \beta^{-1}(E)
% \right)
% \\
% &=
% \operatorname{Leb}(E).
% \end{aligned}
% \]
%Thus 
\(\iota\) also preserves the Lebesgue measure.
\end{proof}
\color{black}

By Lemma~\ref{lemma:ord_isom}, we obtain the following statement for i.i.d.\ random variables on \([0,1)\).

\begin{prop}\label{prop:LIS_succ}
For an i.i.d.\ uniform sequence \(a=(a_1,a_2,\dots,a_n)\) taking values in \([0,1)\), let \(U_n(a)\) denote the length of the LIS taken with respect to the usual order, and let \(U_n^{\succ}(a)\) denote the length of the LIS taken with respect to the order \(\succ\). Then, we have
\[
U_n
\stackrel{d}{=}
U_n^{\succ}.
\]
\end{prop}

Let \(x\in [0,1)\) be a uniformly distributed random variable, and define the sequence \(a_1,a_2,a_3,\dots\in [0,1)\) by setting \(a_1=x\) and $a_{n+1}:=f(a_n)$.
For the first \(N\) terms \(a^{(N)}=(a_1,\dots,a_N)\), we let \[\pi_{\succ}^{(N)}:=\mathrm{Pat}_{\succ}(a^{(N)})\in S_N\] 
be the sequence with respect to the order $\succ$ (see \eqref{eq:def_of_pi}).
% \[
% \pi_{\succ}^{(N)}:=\mathrm{Pat}_{\succ}(a^{(N)}).
% \]
Let \(\lambda_{\succ}^{(N)}\) be the Young diagram obtained by applying Schensted's row-insertion algorithm to \(\pi_{\succ}^{(N)}\), and let \(\lambda_{\succ,1}^{(N)}\) denote the length of its first row.

In this paper, we study the asymptotic behavior of
\(
\EE[\lambda_{\succ,1}^{(N)}]
\)
as \( N \to \infty \).
The main theorem of this paper is the following:
\begin{thm}\label{thm:main2}
Assume that the order \(\succ\) on \([0,1)\) satisfies property \((\ast)\).
Then, for any \(\epsilon>0\), we have
%\marginpar{\tiny YN: as \(N\to\infty\)がちょっとフランク？割って極限取って=2で書くか（$\epsilon$不要）、for any sufficiently large $N$にするとか？}
\[
(2-\epsilon)\sqrt{N}<\mathbb{E}[\lambda_{\succ,1}^{(N)}]<(2+\epsilon)\sqrt{N} \qquad \text{as \(N\to\infty\).}
\]
\end{thm}

% By applying Theorem~\ref{thm:main2}, the original result (Theorem~\ref{thm:main1}) can be extended from the doubling map to maps other than the doubling map, such as the tent map.

If we take $\succ$ as the ordinary order $>$, we obtain the following estimate for the doubling map $f$.
\begin{cor}\label{thm:main1}
For any \( \epsilon > 0 \), we have
\[
(2-\epsilon)\sqrt{N} < \mathbb{E}[\lambda_1^{(N)}] < (2+\epsilon)\sqrt{N}\qquad \text{ as \( N \to \infty \).}
\]
\end{cor}

\subsection{Application to the tent map}

An interesting application   where a non-trivial order $\succ$ appears is the dynamics of the \textit{tent map} \(g:[0,1]\to [0,1]\), defined by
\[
g(x)=1-|1-2x|.
\]
% The graph of $g$ is given in Fig \ref{fig:tentmap}.

% \begin{figure}[htpb]
% \centering
% \begin{tikzpicture}[scale=2]

%   \draw[->] (-0.2,0) -- (1.15,0) node[right] {$x$};
%   \draw[->] (0,-0.2) -- (0,1.15) node[above] {$g(x)$};

%   \draw (0.5,0.02) -- (0.5,-0.02)
%         node[below] {$\frac12$};
%   \draw (1,0.02) -- (1,-0.02)
%         node[below] {$1$};
%   \draw (0.02,1) -- (-0.02,1)
%         node[left] {$1$};

%   \node[below left] at (0,0) {$O$};

%   \draw[dashed,gray] (0.5,0) -- (0.5,1);
%   \draw[dashed,gray] (0,1) -- (0.5,1);

%   \draw[very thick]
%     (0,0) -- (0.5,1) -- (1,0);

%   \fill (0.5,1) circle (0.6pt);

%   % \node[right] at (0.9,0.68)
%   %   {$g(x)=1-\lvert 1-2x\rvert$};

% \end{tikzpicture}
%     \caption{The graph of the tent map $g(x)=1-|1-2x|$.}
%     \label{fig:tentmap}
% \end{figure}

The sequence generated by $g$ can be encoded by the \textit{Gray code}~\cite{Arroyo2014} explained below.
For a binary code $x=c_1c_2c_3\dots$, the \textit{Gray code} associated to $x$ is given by
\[
\widetilde{\Gamma}(x)=w_1w_2w_3\dots,
\]
where
\[
w_n \equiv c_n+c_{n-1}\pmod 2,\qquad w_n\in\{0,1\},
\]
and \(c_0=0\).
The Gray code transformation $\widetilde{\Gamma}$ is invertible, and the inverse map $\widetilde{\Gamma}^{-1}(w)=0.c_1c_2c_3\dots$ is given as
\begin{equation}\label{eq:example_of_order}
c_i\equiv w_1+w_2+\dots+w_i \pmod{2}.
\end{equation}

Let $\succ_G$ be the \textit{Gray code order} defined as follows:
For any two distinct binary codes 
\[
w=w_1w_2\dots w_iw_{i+1}w_{i+2}\dots\,\text{ and }\,
w'=w_1w_2\dots \overline{w_i}w'_{i+1}w'_{i+2}\dots,
\]
which first differ at the $i$-th digit, define
\begin{equation}\label{eq:def_of_G_order}
\begin{aligned}
&w
\succ_G
w'\iff w_1+w_2+\dots+w_i\equiv 1\pmod 2.
\end{aligned}
\end{equation}
Comparing \eqref{eq:def_of_G_order} with \eqref{eq:example_of_order}, we have
\begin{equation}\label{eq:til_Gamma_rel}
x>y\iff \widetilde{\Gamma}(x)\succ_G\widetilde{\Gamma}(y)
\end{equation}
for any $x,y\in \Omega$.

Let $\widetilde{g}:\Omega\to \Omega$ be the map defined by
\[
\tilde{g}:
c_1c_2c_3\dots \mapsto 
\begin{cases}
c_2c_3c_4\dots  & (c_1=0),\\
\overline{c_2}\, \overline{c_3}\,\overline{c_4}\dots  & (c_1=1).
\end{cases}
\]
\begin{lemma}\label{lemma:comm_rel_gG}
The commutation relation
$
\widetilde{\Gamma}\circ \tilde{g}=\tilde{f}\circ \widetilde{\Gamma}
$    
holds, where $\widetilde{f}$ is the shift map \eqref{eq:def_shift}.
\end{lemma}
\begin{proof}
Let $x=c_1c_2c_3\dots\in \Omega$,
$\widetilde{\Gamma}\circ \widetilde{g}(x)=a_1a_2a_3\dots$, and 
$\widetilde{f}\circ \widetilde{\Gamma}(x)=b_1b_2b_3\dots$.
Then we have
\[
a_1\equiv 
\begin{cases}
c_2&(c_1=0)\\
\overline{c_2}&(c_1=1)
\end{cases},\ 
a_i\equiv 
\begin{cases}
c_i+c_{i+1} & (c_1=0)\\
\overline{c_i}+\overline{c_{i+1}}&(c_1=1)
\end{cases}
,\qquad b_i\equiv c_i+c_{i+1},
\]
which implies $a_i=b_i$.
\end{proof}

Through the binary coding map $\beta$, the Gray order naturally induces a total order on the interval $[0,1)$.
For brevity, we use the same notation $\succ_G$ to describe that order.
Since the order $\succ_G$ satisfies the property \((\ast)\), it follows from Lemma \ref{lemma:ord_isom} that there exists two Lebesgue-null sets $Z_1,Z_2\subset [0,1)$ and a measure-preserving bijection
\[
\Gamma:Z_1^c\to Z_2^c
\]
satisfying 
\begin{equation}\label{eq:rel_through_G}
X> Y\iff \Gamma(X)\succ_G \Gamma(Y).
\end{equation}

From \eqref{eq:til_Gamma_rel} and \eqref{eq:rel_through_G}, we find that $\Gamma$ is induced from $\widetilde{\Gamma}$ through $\beta$:
\[
\Gamma=\beta\circ \widetilde{\Gamma}\circ(\beta|_{Z_1^c})^{-1}.
\]
On the other hand, on the set $\mathbb{Q}^c:=[0,1)\setminus \mathbb{Q}$, $g$ is induced from $\widetilde{g}$:
\[
g=\beta \circ \widetilde{g}\circ (\beta|_{\mathbb{Q}^c})^{-1}.
\]
By Lemma \ref{lemma:comm_rel_gG}, except for the countable set $Z_1\cup (\mathbb{Q}\cap [0,1))$, $\Gamma$ satisfies the commutative relation
\begin{equation}\label{eq:comm_rel_G}
\Gamma\circ g=f\circ \Gamma.
\end{equation}

\begin{lemma}\label{lemma:f_pattern_g}
Except for countably many $x\in [0,1)$, the order pattern of
\[
x,g(x),\dots,g^{N-1}(x)
\]
in the usual order agrees with the order pattern of
\[
\Gamma(x),f(\Gamma(x)),\dots,f^{N-1}(\Gamma(x))
\]
in the order \(\succ_G\).
\end{lemma}
\begin{proof}
This lemma follows from \eqref{eq:rel_through_G} and \eqref{eq:comm_rel_G}.
\end{proof}

\begin{example}
For $x=0.11000101...$, the Gray code expression of $x$ is
\[
\Gamma(x)=0.10100111...
\]
One checks that $g(x)=0.0111010...$, and its Gray code expression is 
\[
\Gamma(g(x))=0.0100111...
=f(\Gamma(x)).
\]
\end{example}

It follows from Lemma \ref{lemma:f_pattern_g} that the expectation value of the length of a $>$-LIS obtained by the tent map is the same as that of a $\succ$-LIS obtained by the doubling map.
Hence, we obtain the following corollary from Theorem~\ref{thm:main2}.

\begin{cor}
Let $E_N$ be the expectation value of the length of a LIS contained in the sequence $x,g(x),g^2(x),\dots,g^{N-1}(x)$ with respect to the Lebesgue measure.
For any $\epsilon>0$, we have
\[
(2-\epsilon)\sqrt{N}<E_N<(2+\epsilon)\sqrt{N}
\]
as $N\to \infty$.
\end{cor}

\subsection{LIS and finite cylinder partitions}

A finite set \(\mathcal W\) of finite binary words is called a
\emph{complete prefix set} if the cylinders
\[
\{C_u:u\in\mathcal W\}
\]
are pairwise disjoint and cover \([0,1)\).  Let $\succ$ be a total order on $[0,1)$ satisfying the property \((\ast)\).
By Lemma \ref{lemma:cylinder_property}, if \(w,w'\in\mathcal W\) are distinct, exactly one of $C_w\prec C_{w'}$ and $C_{w'}\prec C_w$ holds.
Thus, after enumerating \(\mathcal W=\{w_0,\dots,w_{R-1}\}\)
appropriately and setting \(I_r:=C_{w_r}\), we may assume that
\[
I_0\prec I_1\prec\cdots\prec I_{R-1},\quad (R=\# \mathcal{W}).
\]

For a finite sequence \(z=(z_1,\dots,z_M)\), write
\(\mathrm{LIS}_\succ(z)\) for the length of a longest $\succ$-increasing subsequence contained in $z$.

%\input{figures/fig_partition}
\input{fig_partition}

\begin{prop}\label{prop:decomp}
Let \(a=(a_1,\dots,a_N)\in[0,1)^N\), and let
\begin{equation}\label{eq:order_separation}I_0\prec I_1\prec\cdots\prec I_{R-1}
\end{equation}
be the cylinder partition associated with a complete prefix set.
Let
\[
\Delta_N^{(R)}
:=
\left\{
(t_0,\dots,t_R)\in\mathbb Z^{R+1}:
1=t_0\leq t_1\leq\cdots\leq t_R=N+1
\right\}
\]
be the set of partitions of $\{1,2,\dots,N+1\}$.
Then, we have
\begin{equation}\label{eq:lambda_est}
\lambda_{\succ,1}^{(N)}
=
\max_{(t_0,\dots,t_R)\in\Delta_N^{(R)}}
\sum_{r=0}^{R-1}
\mathrm{LIS}_\succ
\bigl(
a_i:t_r\leq i<t_{r+1},a_i\in I_r
%I_r\cap \{a_{t_r},a_{t_r+1},\dots,a_{t_{r+1}-1}\}
\bigr).
\end{equation}
\end{prop}

\begin{proof}
Let
\[
a_{\kappa_1}\prec a_{\kappa_2}\prec\cdots\prec a_{\kappa_\ell},
\qquad
1\leq\kappa_1<\cdots<\kappa_\ell\leq N,
\]
be a longest \(\succ\)-increasing subsequence, so that
\(\ell=\lambda_{\succ,1}^{(N)}\).  For each \(j\), let
\(\rho_j\in\{0,\dots,R-1\}\) be the unique index satisfying
\(a_{\kappa_j}\in I_{\rho_j}\).  By
\eqref{eq:order_separation}, the sequence
\(\rho_1,\dots,\rho_\ell\) is nondecreasing.

For \(0\leq r<R\), define
\[
t_r
:=
1+\max\bigl(\{\kappa_j:\rho_j<r\}\cup\{0\}\bigr),
\]
and set \(t_R:=N+1\).  Then
\((t_0,\dots,t_R)\in\Delta_N^{(R)}\).  If \(\rho_j=r\), then
\(
t_r\leq\kappa_j<t_{r+1}.
\)
Consequently, the part of the chosen LIS lying in \(I_r\) is contained
in
\[
(a_i:t_r\leq i<t_{r+1},\ a_i\in I_r)
=
I_r\cap \{a_{t_r},a_{t_r+1},\dots,a_{t_{r+1}-1}\}.
\]
Summing over \(r\) proves that the left-hand side of
\eqref{eq:lambda_est} is at most its right-hand side.

Conversely, fix
\((t_0,\dots,t_R)\in\Delta_N^{(R)}\), and for every \(r\) choose a
longest \(\succ\)-increasing subsequence of
\(
(a_i:t_r\leq i<t_{r+1},\ a_i\in I_r).
\)
The time intervals occur in increasing order, and every point of
\(I_r\) is \(\succ\)-smaller than every point of \(I_{r+1}\).
Therefore the concatenation of these \(R\) subsequences is a
\(\succ\)-increasing subsequence of \(a\).  This proves the reverse
inequality.
\end{proof}
\color{black}

\section{Lower Bound}

% Hereafter, we fix a total order \(\succ\) satisfying the property \((\ast)\).
In this section, we evaluate a lower bound of the expectation value $\EE[\lambda^{(N)}_{\succ,1}]$.
The basic idea of the proof is to find a ``sufficiently large'' subsequence of $a_1,a_2,\dots$ distributed like a uniform random permutation.

\subsection{Uniform complete prefix set $\mathcal{W}_m$}

For \(m\geq 1\), let $\mathcal{W}_m:=\{0,1\}^m$ be the set of binary words of length $m$.
By the property \eqref{eq:property_one}, one can arrange the element of \(\mathcal W_m\) as
\[
w_0,w_1,\dots,w_{R-1},\qquad (R=2^m)
\]
so that the set $\{C_{w}:w\in \mathcal{W}_m\}$ satisfies $C_{w_0}\prec C_{w_1}\prec\cdots\prec C_{w_{R-1}}$.
Therefore, \(\mathcal{W}_m\) is a complete prefix set.
Set \(I_r:=C_{w_r}\).
% , so
% \(I_0,\dots,I_{R-1}\) form a partition of \([0,1)\).

\begin{example}
Let \(m=3\) and let \(\succ\) be the total order defined in \eqref{eq:def_of_G_order}.  Then
\[
C_{000}\prec C_{001}\prec
C_{011}\prec
C_{010}\prec
C_{110}\prec
C_{111}\prec
C_{101}\prec
C_{100},
\]
and hence
\[
\begin{gathered}
I_0=\left[0,\frac18\right),\quad
I_1=\left[\frac18,\frac28\right),\quad
I_2=\left[\frac38,\frac48\right),\quad
I_3=\left[\frac28,\frac38\right),\\
I_4=\left[\frac68,\frac78\right),\quad
I_5=\left[\frac78,1\right),\quad
I_6=\left[\frac58,\frac68\right),\quad
I_7=\left[\frac48,\frac58\right).
\end{gathered}
\]
\end{example}

Let
\[
T:=\left\lfloor\frac{N}{R}\right\rfloor,
\qquad
J_r:=\{rT+1,rT+2,\dots,(r+1)T\}
\quad(0\leq r<R).
\]
Let \(L_r\) be the \(\succ\)-LIS length of
\[
(a_i:i\in J_r,\ a_i\in I_r).
\]
Choose \(t_r=rT+1\) for \(0\leq r<R\) and \(t_R=N+1\) in
Proposition~\ref{prop:decomp}.  The last time interval may contain the
unused remainder \(\{RT+1,\dots,N\}\), which can only increase its LIS.
Therefore
\begin{equation}\label{eq:lambda1_vs_L}
\lambda_{\succ,1}^{(N)}
\geq
\sum_{r=0}^{R-1}L_r.
\end{equation}
\color{black}

\subsection{Random variable $K_r$}
Let \(x\in [0,1)\) be the initial value, and
\begin{equation}\label{eq:binary}
x=0.b_1b_2b_3\cdots\qquad (b_i\in \{0,1\})
\end{equation}
be its binary expansion.
Since \(x\) is uniformly distributed, the binary sequence \(b_1,b_2,b_3,\dots\in \{0,1\}\) is an i.i.d. sequence.

We take three positive integers $S,t$, and $B$ satisfying
\[
S\geq m\quad \text{and}\quad B=S+t.
\]
For now, we leave these parameters unspecified; their precise dependence on \(N\) will be fixed later.
We divide the first \(N\) digits of the binary expansion \eqref{eq:binary} into $R$ blocks of length $T$ (see Example \ref{ex:first} below).
Each block is divided into $\lfloor T/B\rfloor$ subblocks of length \(B\).
(In these procedures, remainders are ignored.)
In addition, each subblock is divided into two parts: the first part consists of $S$ consecutive digits, and the latter part consists of $t$ consecutive digits.
We call the first part of each subblock the \emph{search window}, and the latter part the \emph{tail}.
Let \(J_{r,i}\) denote the \(i\)-th subblock contained in \(J_r\), namely
\[
J_{r,i}:=\{rT+(i-1)B+1,rT+(i-1)B+2,\dots,rT+iB\}
\subset J_r.
\]

\begin{example}\label{ex:first}
If $N=80$, $m=3$, $R=2^m=8$, and $(S,t,B)=(3,1,4)$, the binary expansion of $x$ is decomposed as
\[
x=0.
\big|\overbrace{(\underbrace{110|0}_{J_{0,1}})(\underbrace{101|0}_{J_{0,2}})10}^{J_0}
\big|\overbrace{(\underbrace{001|0}_{J_{1,1}})(\underbrace{001|1}_{J_{1,2}})10}^{J_1}
\big|\overbrace{(\underbrace{110|0}_{J_{2,1}})(\underbrace{010|0}_{J_{2,2}})11}^{J_2}
\big|
\cdots
\big|\overbrace{(\underbrace{000|0}_{J_{7,1}})(\underbrace{101|0}_{J_{7,2}})10}^{J_7}\big| 0001...
\]
Each block has length $10$, and each subblock has length $4$.
Each subblock has a search window of length $3$ and a tail of length $1$.
% The boundaries between the blocks are indicated by thick vertical lines, the boundaries between the subblocks by brackets, and the first and second halves of each subblock are separated by thin vertical lines.
\end{example}

% Let $w_r \in \{0,1\}^m$ be the binary expansion of $0 \leq s(r) < 2^m$.
We will say ``\textit{$J_{r,i}$ hits $w_r$}'' if the search window of $J_{r,i}$ contains a consecutive subsequence that coincides with $w_r$.
In Example~\ref{ex:first}, $J_{1,1}$ and $J_{1,2}$ hit $001$, and $J_{2,2}$ hits $010$.

Recall that the doubling map $f$ corresponds to the shift of the binary expansion.
Then, we have
\begin{equation}\label{eq:cond_Jri}
\text{$J_{r,i}$ hits $w_r$} \iff 
\text{there exists at least one $1\leq j\leq S-m+1$ such that 
$a_{rT+(i-1)B+j}\in I_r$
}.
\end{equation}

Let $K_r$ denote the number of subblocks contained in $J_r$ that hit $w_r$:
\[
K_r:=\#\left\{1\leq i\leq\left\lfloor T/B\right\rfloor:J_{r,i}\text{ hits }w_r\right\}.
\]
Here \(K_r\) is a random variable on \([0,1)\) with respect to Lebesgue measure; its dependence on \(x\) and \(w_r\) is suppressed from the notation.
%(Note that $K_r=K_r(x)$ is a random variable on the probability space $[0,1)$ equipped with the Lebesgue measure, which also depends on $w_r$, while we suppress them from the notation for simplicity.)
%\marginpar{\tiny 英語}
Because the search windows of distinct subblocks are disjoint and the digits \(b_i\) are independent, the corresponding hit events are independent.
Therefore, the random variable $K_r$ follows the binomial distribution 
\begin{equation}\label{eq:bin_dist}
K_r \sim \mathrm{Bin}(\lfloor T/B \rfloor, p_r),
\end{equation}
where $p_r$ is the probability
\begin{equation}\label{eq:def_of_p}
p_r = \frac{\#(\text{binary words of length $S$ hitting $w_r$})}{2^S}.
\end{equation}

To evaluate the probability $p_r$ more precisely, we define the \textit{period} of binary words.
% \begin{equation}
% \text{$J_{r,i}$ hits $w_r$} \iff 
%  1 \leq \exists j \leq B \ \text{s.t.}\ 
% (a_{rT+(i-1)B+j} \in I_r).
% \end{equation}
\begin{defi}
For a binary word
$w=b_1b_2\cdots b_m$,
define its (minimal) period by
\[
\chi(w):=\min\left\{1\leq c\leq m: b_i=b_{i+c} \text{ for every }1\leq i\leq m-c\right\}.
\]
In particular, we have \(\chi(w)=m\) if \(w\) has no period smaller than its length.
\end{defi}

% For brevity, we write
% \[
% C(r):=C(w_r).
% \]

\begin{lemma}\label{lemma:ineq}
Let $\chi=\chi(w_r)$ be the period of a binary word $w_r\in \mathcal{W}_m$.
Then, the probability $p_r$ defined in \eqref{eq:def_of_p} satisfies the inequality
\[
p_r \geq \frac{(S-m)}{2^m} - \frac{(S-m)^2}{2^{m+\chi}}.
\]
\end{lemma}

\begin{proof}
Let
\(
A = \{ \text{binary words of length $S$ that hit $w_r$}\} \).
Then we have $p_r=|A|/2^S$.
For each $1\leq i\leq S-m+1$, we also define
\[
A_i = \left\{
\begin{aligned}
&\text{binary words of length $S$ whose entries}\\ 
&\text{from the $i$-th to the $(i+m-1)$-th positions is $w_r$}
\end{aligned}
\right\}.
\]
By the inclusion-exclusion principle, we have
\[\begin{aligned}
|A| &\geq \sum_{i=1}^{S-m+1} |A_i| - \sum_{1 \leq i < j \leq S-m+1} |A_i \cap A_j|\\
&= (S-m+1)2^{S-m} - \sum_{i<j} |A_i \cap A_j|.
\end{aligned}
\]
If $\chi\leq j-i<m$, the two occurrences of $w_r$ prescribe the digits
on the union of their positions, which consists of $m+(j-i)$ numbers,
and hence
\[
|A_i \cap A_j| \leq 2^{S-m-(j-i)} \leq 2^{S-m-\chi}.
\]
If $j-i\geq m$, the two occurrences prescribe $2m$ digits at disjoint
positions, and hence
\[
|A_i \cap A_j| = 2^{S-2m} \leq 2^{S-m-\chi},
\]
where the last inequality follows from $\chi\leq m$.
Thus, in either case, we have $|A_i\cap A_j|\leq 2^{S-m-\chi}$.

Since the number of pairs $(i,j)$ satisfying $j - i \geq \chi$ and $1 \leq i < j \leq S-m+1$ is $\binom{S-m-\chi+2}{2}$, we obtain
\[
\begin{aligned}
|A|
&\geq (S-m+1)2^{S-m} - \binom{S-m-\chi+2}{2} 2^{S-m-\chi} \\
&\geq (S-m)2^{S-m} - (S-m)^2 2^{S-m-\chi}.
\end{aligned}
\]
From this, we conclude the desired inequality.
\end{proof}

\begin{lemma}\label{lemma:ev_of_EK}
Suppose that
\[
2^{\chi(w_r)}\gg S,
\qquad
S\gg m,
\qquad
N\gg 2^mB,
\qquad
NS\gg 4^mB.
\]
Then
\[
\EE[K_r]
\geq
(1-o(1))\frac{NS}{4^mB},
\qquad
\EE[\sqrt{K_r}]
\geq
(1-o(1))
\frac{\sqrt{NS}}{2^m\sqrt B},
\qquad
\EE[K_r^3]
\leq
\left\lfloor\frac{T}{B}\right\rfloor^3.
\]
\end{lemma}

\begin{proof}
By Lemma~\ref{lemma:ineq} and the assumptions,
\[
\begin{aligned}
p_r
\geq
\frac{S-m}{2^m}
-
\frac{(S-m)^2}{2^{m+\chi(w_r)}} 
=
\frac{S-m}{2^m}
\left(
1-\frac{S-m}{2^{\chi(w_r)}}
\right)
=
(1-o(1))\frac{S}{2^m}.
\end{aligned}
\]
Since \(T=\lfloor N/2^m\rfloor\), 
%and \(N\gg 2^mB\),
it follows that
\[
\EE[K_r]
=
\left\lfloor\frac{T}{B}\right\rfloor p_r
\geq
(1-o(1))\frac{NS}{4^mB}.
\]
In particular, the last assumption implies that
\(\EE[K_r]\to\infty\).

Since \(K_r\) is binomial,
\[
\mathrm{Var}(K_r)\leq\EE[K_r].
\]
Moreover, for \(x,y\geq0\),
\[
|\sqrt{x}-\sqrt{y}|
\leq
\sqrt{|x-y|}.
\]
Therefore,
\[
\begin{aligned}
\EE[\sqrt{K_r}]
&\geq
\sqrt{\EE[K_r]}
-
\EE\left[
\left|
\sqrt{K_r}-\sqrt{\EE[K_r]}
\right|
\right] \\
&\geq
\sqrt{\EE[K_r]}
-
\mathrm{Var}(K_r)^{1/4} \\
&\geq
\sqrt{\EE[K_r]}
-
\EE[K_r]^{1/4} \\
&=
(1-o(1))\sqrt{\EE[K_r]} \\
&\geq
(1-o(1))
\frac{\sqrt{NS}}{2^m\sqrt B}.
\end{aligned}
\]
Finally,
$
0\leq K_r\leq
\left\lfloor\frac{T}{B}\right\rfloor
$
implies
$
\EE[K_r^3]
\leq
\left\lfloor\frac{T}{B}\right\rfloor^3.
$
\end{proof}

\subsection{Comparison with the uniform permutation}\label{sec:lower_uniform_comparison}

For each $0\leq r<R=2^m$, define the random set \(\mathcal K_r\) by
\[
\mathcal K_r
:=
\left\{
1\leq i\leq \left\lfloor T/B\right\rfloor:
J_{r,i}\text{ hits }w_r
\right\}.
\]
Recall $K_r=\#\mathcal K_r$.
%We suppress the dependence of the objects below on the initial point \(x\).
For each \(i\in\mathcal K_r\), let \(\tau(r,i)\) be the first time at
which \(w_r\) occurs in the search window of \(J_{r,i}\).
More precisely,
\[
\tau(r,i)
:=
\min\left\{
\tau:
rT+(i-1)B+1\leq \tau
\leq rT+(i-1)B+S-m+1,\quad
a_\tau\in I_r
\right\}.
\]

Let \(y_{r,i}\) be the binary word obtained by taking the first \(m+t\) digits of \(a_{\tau(r,i)}\).
By definition $\tau(r,i)$, the first \(m\) digits of $y_{r,i}$ must coincide with \(w_r\).
We denote its last \(t\) digits by $\eta_{r,i}$.
Then, $y_{r,i}$ is the concatenation of $w_r$ and $\eta_{r,i}$.
% \[
% \eta_{r,i}
% :=
% b_{\tau(r,i)+m}
% b_{\tau(r,i)+m+1}
% \cdots
% b_{\tau(r,i)+m+t-1}.
% \]
Since \(B=S+t\), the whole word \(y_{r,i}\) is contained in the
subblock \(J_{r,i}\).

Seeing the initial value $x\in [0,1)$ as a Bernoulli digit, the stopping time for the event $i\in \mathcal{K}_r$ is \(\tau(r,i)+m-1\).
%for the Bernoulli digits in this subblock.
Moreover, conditional on \(\mathcal K_r\) and on the first hitting times \((\tau(r,i))_{i\in\mathcal K_r}\), the tail words
$
(\eta_{r,i})_{i\in\mathcal K_r}
$
are independent and uniformly distributed on \(\{0,1\}^{t}\).

When the tail words \(\eta_{r,i}\) are pairwise
distinct for each $i\in \mathcal{K}_r$, the relative order with respect to \(\succ\) of the values
\(a_{\tau(r,i)}\) is determined by \(\eta_{r,i}\).
Otherwise, deeper digits may be needed to compare some of the selected values.
We denote this ``bad event'' by
\[
\mathrm{Coll}_r
:=
\left\{
x\in[0,1):
\exists\,i<j,\quad
i,j\in\mathcal K_r(x),\quad
\eta_{r,i}(x)=\eta_{r,j}(x)
\right\}.
\]

\begin{lemma}\label{lemma:birth}
For $0\leq k\leq \lfloor T/B\rfloor$, we have
\[
\mathbb{P}(\mathrm{Coll}_r\mid K_r=k)\leq\binom{k}{2}2^{-t}.
\]
\end{lemma}

\begin{proof}
If $K_r=k$, the conditional probability that the complementary event $\mathrm{Coll}^c_r$ occurs is
\[
\mathbb{P}(\mathrm{Coll}^c_r\mid K_r=k)=\frac{2^{t}(2^{t}-1)(2^t-2)\cdots (2^{t}-k+1)}{(2^t)^k}\geq 1-\binom{k}{2}2^{-t}.
\]
The desired inequality follows from this.
\end{proof}

We order the elements of $\mathcal{K}_r$ in increasing order as
\[
\mathcal{K}_r=\{
i_{r,1}<i_{r,2}<\dots<i_{r,K_r}
\}.
\]
Then, from the above discussion, we obtain the following proposition.

\begin{prop}\label{prop:near_unif}
Under $\mathrm{Coll}^c_r$, the distribution of the relative order with respect to $\succ$ of the sequence
\[
a_{\tau(r,i_{r,1})},a_{\tau(r,i_{r,2})},\dots,a_{\tau(r,i_{r,K_r})}
\]
coincides with the distribution of the relative order of a uniform random sequence of length $K_r$.
\end{prop}

By Proposition~\ref{prop:LIS_succ} and Proposition~\ref{prop:near_unif}, we conclude that, under $\mathrm{Coll}^c_r$, the expected length of the $\succ$-LIS contained in
\begin{equation}\label{eq:pickuped_sec}
a_{\tau(r,i_{r,1})},a_{\tau(r,i_{r,2})},\dots,a_{\tau(r,i_{r,K_r})}
\end{equation}
coincides with that for a uniform random permutation of length $K_r$.

\subsection{A lower bound of $\EE[\lambda_{\succ,1}^{(N)}]$}

In what follows, we use the following known theorem.

\begin{thm}\label{thm:known}
For a uniform random permutation $\pi\in S_n$, let $U_n(\pi)$ denote the length of the LIS contained in $\pi$.
Then, for any $\epsilon>0$, there exist some positive constants $c(\epsilon),n(\epsilon)>0$ such that
\[
n\geq n(\epsilon)\ \Rightarrow\ 
\PP\left(\left|\frac{U_n}{\sqrt{n}}-2\right|>\epsilon\right)<\exp(-c(\epsilon)\sqrt{n}).
\]
In particular, as $n\to \infty$, we have $(2-\epsilon)\sqrt{n}+o(1)<\EE[U_n]<(2+\epsilon)\sqrt{n}+o(1)$.
\end{thm}
\begin{proof}
This follows from the large deviation estimates of
Deuschel and Zeitouni~\cite{DeuschelZeitouni1999}.
Indeed, their Theorem~1 gives the lower-tail estimate
\[
\lim_{n\to\infty}\frac1n
\log \PP\{U_n<(2-\epsilon)\sqrt n\}<0,
\]
while their Theorem~2 gives the upper-tail estimate
\[
\lim_{n\to\infty}\frac1{\sqrt n}
\log \PP\{U_n>(2+\epsilon)\sqrt n\}<0.
\]
Combining these two bounds yields the stated estimate.
\end{proof}

Let $W_r$ denote the length of the $\succ$-LIS contained in the sequence \eqref{eq:pickuped_sec}.
By Proposition~\ref{prop:near_unif}, we have
\[
\mathbb{E}[W_r\mid K_r=k]\geq \mathbb{E}[U_k]\cdot\mathbb{P}(\mathrm{Coll}^c_r\mid K_r=k).
\]
This inequality leads
\begin{align*}
\mathbb{E}[W_r]
&\geq \sum_{k} \mathbb{E}[U_k]
\cdot\mathbb{P}(\mathrm{Coll}^c_r\mid K_r=k)\cdot 
\mathbb{P}(K_r=k)\\
&\geq \sum_{k} \mathbb{E}[U_k]\cdot 
\left(1-\binom{k}{2}2^{-t}\right)
\mathbb{P}(K_r=k)\qquad \text{(Lemma \ref{lemma:birth})}\\
&\geq 
\sum_k\EE[U_k]\PP(K_r=k)-2^{-t}\sum_k k\cdot \binom{k}{2}\mathbb{P}(K_r=k)\\
&\geq 
\sum_k\EE[U_k]\PP(K_r=k)-2^{-t}\mathbb{E}[K_r^3].
\end{align*}
On the other hand, from Theorem~\ref{thm:known} and the trivial bound $U_k \le k$, we have
%\marginpar{\tiny $\EE[\sqrt{K_r}]\to \infty$を使いたい}
\begin{align*}
\sum_k\EE[U_k]\PP(K_r=k)
&>(2-\epsilon)\sum_k \sqrt{k}\PP(K_r=k) + \sum_{0\le k\le n(\epsilon)}\left(\EE[U_k]-(2-\epsilon)\sqrt{k}\right)\PP(K_r=k) \\
&\ge (2-\epsilon)\EE[\sqrt{K_r}] -n(\epsilon)^2
\end{align*}
for any $\epsilon>0$.
From these inequalities, we derive
\begin{equation}\label{eq:EE[U]}
\EE[W_r]\geq (2-\epsilon)\EE[\sqrt{K_r}]-2^{-t}\EE[K_r^3] -n(\epsilon)^2.    
\end{equation}

Let $L_r$ be the random variable in \eqref{eq:lambda1_vs_L}, the length of the $\succ$-LIS contained in the subsequence
\[
(a_i:i\in J,a_i\in I_r)=
I_r\cap \{a_{rT+1},a_{rT+2},\dots,a_{(r+1)T}\}
.\]
Since this subsequence contains the sequence \eqref{eq:pickuped_sec}, we have $L_r \geq W_r$, which implies $\mathbb{E}[L_r]\geq \mathbb{E}[W_r]$.
% By \eqref{eq:EE[U]}, under the condition $2^{C(r)} \gg S$, we obtain
% \begin{align*}
% \mathbb{E}[L_r]
% &\geq 
% \mathbb{E}[W_r]\\
% &\geq \mathbb{E}[(2-\epsilon)\sqrt{K_r}]-2^{-t}\mathbb{E}[K_r^3]\\
% &\geq (2-\epsilon)(1-o(1))\sqrt{\mathbb{E}[K_r]}-2^{-\ell}\mathbb{E}[K_r^3].
% \end{align*}
Therefore, from \eqref{eq:lambda1_vs_L} and \eqref{eq:EE[U]}, we have
\begin{equation}\label{eq:Elam_vs_E}
\mathbb{E}[\lambda_{\succ,1}^{(N)}]\geq \sum_{r=0}^{R-1}\mathbb{E}[L_r]\geq 
\sum_{r=0}^{R-1}\mathbb{E}[W_r]
\geq 
\sum_{r=0}^{R-1}
\left(
(2-\epsilon)\EE[\sqrt{K_r}]-2^{-t}\EE[K_r^3]-n(\epsilon)^2
\right).
\end{equation}
% which gives a computable lower bound of $\mathbb{E}[\lambda_{\succ,1}^{(N)}]$.

\subsection{Evaluation of the lower bound as $N\to \infty$}

We then consider the behavior of the value
\[
\mathcal{E}:=
\sum_{r=0}^{R-1}
\left(
(2-\epsilon)\EE[\sqrt{K_r}]-2^{-t}\EE[K_r^3]-n(\epsilon)^2
\right)
\]
as $N\to \infty$.

Let $P$ be an arbitrary integer satisfying $1\leq P\leq m$.
For any $p$ equal to or smaller than $m$, the number of binary words of length $m$ of period $p$ is at most $2^p$.
Therefore, the number of binary words of length $m$ whose period is \textit{smaller than} $P$ is at most
\(
2+2^2+\dots+2^{P-1}<2^{P}
\).
% \color{blue}
% Therefore, under the condition $2^P\gg S\gg m$ and $T\gg B$, the following inequality holds by Lemma~\ref{lemma:ev_of_EK}:
%\color{red}
Therefore, under the conditions
\[
2^P\gg S\gg m,
\qquad
T\gg B,
\qquad
NS\gg 4^mB,
\]
the following inequality holds by
Lemma~\ref{lemma:ev_of_EK}:
\color{black}
\begin{equation}\label{eq:ev_LL}
\begin{aligned}
\mathcal{E}
&\geq \sum_{w_r\,:\,\chi(w_r)\geq P}
(2-\epsilon)\EE[\sqrt{K_r}]-
\sum_{r=0}^{R-1} 2^{-t}\EE[K_r^3]
-\sum_{r=0}^{R-1}n(\epsilon)^2
\\
&\geq (2-\epsilon)(1-o(1))
\sum_{w_r\,:\,\chi(w_r)\geq P}
\sqrt{\mathbb{E}[K_r]}-2^{-t}\sum_r\mathbb{E}[K^3_r]
-Rn(\epsilon)^2
\\
&\sim  
(2-\epsilon)(1-o(1))
(2^m-2^P)
\frac{\sqrt{NS}}{2^{m}\sqrt{B}}
-2^{-t}\sum_r\mathbb{E}[K^3_r]
-Rn(\epsilon)^2\\
&\geq 
(2-\epsilon)(1-o(1))
(2^m-2^P)
\frac{\sqrt{NS}}{2^{m}\sqrt{B}}
-(1+o(1))2^{m-t}\left(\frac{N}{2^mB}\right)^3
-2^mn(\epsilon)^2.
\end{aligned}
\end{equation}

% \color{blue}
% If we choose the parameters as
% \[
% m \sim \frac{1}{2}\log_2 N,\quad 
% S \sim N^{1/10},\quad 
% \ell \sim 100\log_2 N,\quad 
% P \sim \frac{1}{4}\log_2 N\qquad (N\to \infty),
% \]
% then the condition $2^P \gg S \gg m$ is satisfied. In addition, we have
% \[
% B \sim N^{1/10},\quad 
% 1 \leq P \leq m,\quad 
% \]
% \[
% T=\left\lfloor \frac{N}{2^m} \right\rfloor = O(N^{1/2}) \to \infty,\quad 
% \left\lfloor \frac{T}{B} \right\rfloor = O(N^{2/5}) \to \infty.
% \]
% In this situation, we can evaluate $\mathcal{E}$ by using \eqref{eq:ev_LL} as
% \[
% \mathcal{E}
% \geq (2-\epsilon)\sqrt{N}-o(1)
% \sim (2-\epsilon)\sqrt{N}.
% \]
% From \eqref{eq:Elam_vs_E}, we obtain the desired lower bound.

If we choose the parameters as
\begin{equation}\label{eq:20}
m\sim \frac25\log_2N,
\quad
P\sim \frac14\log_2N,
\quad
S\sim N^{1/10},
\quad
t\sim 100\log_2N,
\quad
B=S+t,
\end{equation}
we have
\[
2^P\gg S\gg m,
\quad
\frac{T}{B}\sim N^{1/2},
\quad
\frac{NS}{4^mB}
\sim
N^{1/5}
\]
as $N\to \infty$.
Thus Lemma~\ref{lemma:ev_of_EK} applies uniformly to every
\(r\) satisfying \(\chi(r)\geq P\).
From
\[
\frac{2^P}{2^m}\to0,
\quad
\frac{S}{B}\to1,
\quad
2^{m-t}
\left(\frac{N}{2^mB}\right)^3
=o(\sqrt N),\quad
2^mn(\epsilon)^2=O(N^{2/5})=o(\sqrt{N})
\]
and \eqref{eq:ev_LL}, we obtain
\[
\mathcal E\geq(2-\epsilon-o(1))\sqrt N,
\]
which proves the desired lower bound by \eqref{eq:Elam_vs_E}.

\section{Upper bound}

To obtain an upper bound, we need a more careful argument.
We will define a complete prefix set $\mathcal{W}$ in a more refined way.

\subsection{Definition of the complete prefix set $\mathcal{W}(P,Q)$}

Fix natural numbers $1<P<Q$, which will serve as thresholds.
A binary word $b_1b_2\ldots b_\ell$ of length $\ell\geq P$ is said to have a \emph{short period} if it has a period smaller than $P$.

For a binary word $w$, let $\len(w)$ denote the length of $w$.
\begin{defi}\label{defi:W}
Let $\mathcal{W}(P,Q)$ denote the set of binary sequences $w$ satisfying one of the following properties~(i)--(iii):
\begin{enumerate}
\item[(i)] $\len(w)=P$, and $w$ has no short period.
\item[(ii)] 
$P<\len(w)\leq Q$, $w$ has no short period, but the subsequence obtained by deleting its last digit has a short period.
\item[(iii)] 
$\len(w)=Q$, and $w$ has a short period.
\end{enumerate}

We write $\mathcal{W}_{\mathrm{regular}}(P,Q)\subset \mathcal{W}(P,Q)$ for the subset consisting of those $w$ satisfying~(i) or~(ii), and $\mathcal{W}_{\mathrm{exceptional}}(P,Q)\subset \mathcal{W}(P,Q)$ for the subset consisting of those $w$ satisfying~(iii).
For a binary code $w\in\mathcal{W}(P,Q)$, denote by $C_w\subset [0,1)$ the cylinder defined in \eqref{eq:def_of_cyl}.
% \[
% I_w=\{x\in [0,1):\text{the first $\len(w)$ digits in the binary expansion of $x$ coincide with $w$}\}.
% \]
\end{defi}

% For a binary word $w$, an \textit{extension of $w$} is a binary word that begins with $w$.
% \begin{equation}\label{eq:w_eq_w}
% x\in C_{w}\cap C_{w'}\text{ for some 
% $x\in [0,1)$ and 
% $w,w'\in \mathcal{W}(P,Q)$, then }
% w=w'.
% \end{equation}

\begin{lemma}\label{lemma:division}
Every element of $[0,1)$ belongs to exactly one cylinder $C_w$ with $w\in \mathcal{W}(P,Q)$.
In other words, $\mathcal{W}(P,Q)$ is a complete prefix set.
\end{lemma}

\begin{proof}
For $x\in [0,1)$, let 
\(
x=0.c_1c_2c_3\dots
\)
be the binary expansion and  $w_m(x):=c_1c_2\dots c_m$ be its first $m$ digits.
Define
% \color{blue}
% \[
% m^\ast:=\min\{m:\text{$x(m)$ has no short period}\}.
% \]
\[
m^\ast
:=
\min\left(
\{m\geq P:w_m(x)\text{ has no short period}\}
\cup\{\infty\}
\right).
\]
Put $w:=w_{\min\{m^\ast,Q\}}(x)$. 
Then, we have $w\in \mathcal{W}(P,Q)$ and $x\in C_w$.
If $m^\ast= P$, then $w$ satisfies (i); if $P<m^\ast\le Q$, then $w$ satisfies (ii); if $m^\ast > Q$, then $w$ satisfies (iii).
To show the uniqueness, note that if a binary word has no short period, its extensions also have no short period.
Therefore, if $w$ is an element of $\mathcal{W}(P,Q)$, any extension of $w$ is not an element of $\mathcal{W}(P,Q)$.
Then, from Lemma \ref{lemma:cylinder_property}, we find that $C_w\cap C_{w'}\neq \emptyset$ for some $w,w'\in \mathcal{W}(P,Q)$ implies $w=w'$.

% The uniqueness follows from \eqref{eq:w_eq_w}.
% \marginpar{\tiny exactly oneの証明についてもう一言コメントできれば嬉しい（が）$\to$最短は$m^\ast$の定義変更とのこと}
% If $m^\ast=P$, then $x$ belongs to some $I_w$ with property~(i).
% If $P<m^\ast<Q$, then $x$ belongs to some $I_w$ with property~(ii).
% If $m^\ast\geq Q$, then $x$ belongs to some $I_w$ with property~(iii).
% To prove uniqueness, observe that if a word has no short period, then
% none of its extensions has a short period. %otherwise, a short period
% % of the extension would also be a short period of the original word.
% %Hence \(\mathcal W\) is prefix-free.
% Therefore, if \(x\in C_u\cap C_v\) for \(u,v\in\mathcal W\), then \(u\) and \(v\)
% are prefixes of the same binary expansion and are therefore comparable under the prefix relation.
% The prefix-free property implies \(u=v\).
\end{proof}

\begin{example}\label{ex:decomp24}
Let $P=2$, $Q=4$, and $\mathcal{W}:=\mathcal{W}(2,4)$. Then, we have
\[
\mathcal{W}_{\mathrm{regular}}=\{01,10,001,110,0001,1110\},\qquad
\mathcal{W}_{\mathrm{exceptional}}=\{0000,1111\}.
\]
In this case,
\[
C_{01}=\left[\frac{1}{4},\frac{1}{2}\right),\quad
C_{10}=\left[\frac{1}{2},\frac{3}{4}\right),
\]
\[
C_{001}=\left[\frac{1}{8},\frac{1}{4}\right),\quad
C_{110}=\left[\frac{3}{4},\frac{7}{8}\right),
\]
\[
C_{0001}=\left[\frac{1}{16},\frac{1}{8}\right),\quad
C_{1110}=\left[\frac{7}{8},\frac{15}{16}\right),\quad
C_{1111}=\left[\frac{15}{16},1\right),\quad
C_{0000}=\left[0,\frac{1}{16}\right).
\]

%\input{figures/fig_word_partition}
\input{fig_word_partition}
\end{example}

\begin{example}\label{ex:decomp36}
Let $P = 3$, $Q = 6$, and $\mathcal{W}:=\mathcal{W}(3,6)$.
Binary sequences with a short period have period $1$ $(000\ldots$, $111\ldots)$ or period $2$ $(0101\ldots$, $1010\ldots)$. In this case, $\mathcal{W}$ consists of 20 words, and $[0,1)$ is divided into 20 intervals (see Figure \ref{fig:word_partition2}).
Notably, the exceptional words in $\mathcal{W}_{\text{exceptional}} = \{000000, 010101, 101010, 111111\}$ are related to
%the words $010101$ and $101010$ sit at 
the period-$2$ orbit $\tfrac{1}{3} = 0.010101\ldots$ and $\tfrac{2}{3} = 0.101010\ldots$ of the doubling map. 
%They generate a fine subdivision near the middle of $[0,1)$.

%\input{figures/fig_word_partition2}
\input{fig_word_partition2}
\end{example}

\begin{example}
Let $P=2$ and $Q=4$ as in Example \ref{ex:decomp24}.
If $\succ$ is the total order for the tent map defined in \eqref{eq:def_of_G_order}, we have
\[
C_{0000}\prec C_{0001}\prec C_{001}\prec C_{01}\prec C_{110}\prec C_{1111}\prec C_{1110}
\prec C_{10}.
\]
\end{example}

% \color{blue}
% \begin{prop}\label{prop:ev_R}
% The inequality
% \[
% R< 2^P(Q-P+2)
% \]
% holds.
% \end{prop}

% \begin{proof}
% The number of the sequences $w$ satisfying~(i) is at most $2^P$.
% A sequence $w$ satisfying~(ii) can be given by its first $P$ digits and length $\len(w)$, because the sequence $w$ except for the last digit has a short period. 
% As $P< \len(w)\leq Q$, the number of such $w$ is at most $2^P(Q-P)$.
% The number of the sequences $w$ satisfying~(iii) is at most $2^P$.
% Then, their total number is at most
% \[
% 2^P+2^P(Q-P)+2^P<2^P(Q-P+2).
% \]
% \end{proof}

% \begin{lemma}\label{lemma:sum_of_powers_of_2}
% The inequality
% \[
% \sum_{w\in \mathcal{W}_{\mathrm{regular}}}2^{\len(w)}<2^{P+Q+1}
% \]
% holds.
% \end{lemma}

% \begin{proof}
% As mentioned in the proof of Proposition \ref{prop:ev_R},  for each $P\leq m\leq Q$, the number of elements $w\in \mathcal{W}_{\mathrm{regular}}$ with $\len(w)=m$ is at most $2^P$.
% Hence, we have
% \[
% \begin{aligned}
% \sum_{w\in \mathcal{W}_{\mathrm{regular}}}2^{\len(w)}
% &\leq 
% 2^P\cdot 2^P+
% \sum_{m=P+1}^{Q}2^{m}\cdot 2^P
% <2^{P+Q+1}.
% % \\
% % &
% % <2^{2P}+2\cdot 2^Q\cdot 2^P
% % <3\cdot 2^{P+Q}.    
% \end{aligned}
% \]
% \end{proof}

%\color{red}

Let $R:=\# \mathcal{W}(P,Q)$ be the size of the complete prefix set $\mathcal{W}(P,Q)$.
In the following, we will use the abbreviation $\mathcal{W}=\mathcal{W}(P,Q)$ if there is no fear of confusion.

\begin{prop}\label{prop:ev_R}
We have
\[
R<2^{P+1}(Q-P+1).
\]
\end{prop}

\begin{proof}
There are at most \(2^P\) words satisfying~(i).
For any $n\geq P$, the number of binary words of length $n$ that have a short period is at most $2+2^2+\cdots+2^{P-1}<2^P$.
Therefore, the number of binary words satisfying (ii) is at most $2\cdot \sum_{n=P}^{Q-1}2^P=2^{P+1}(Q-P)$.
Finally, since the number of binary words satisfying (iii) is at most $2^P$, we conclude
\[
R<2^P+2^{P+1}(Q-P)+2^P=2^{P+1}(Q-P+1).
\]

% Fix \(m\in\{P+1,\dots,Q\}\).
% If a word \(w\) of length \(m\) satisfies~(ii), then the prefix obtained by deleting its last digit has a period \(c<P\).
% For each fixed \(c\), such a prefix is determined by its first \(c\)
% digits and hence has at most \(2^c\) possibilities.
% Therefore, the number of possible prefixes is at most
% \[
% \sum_{c=1}^{P-1}2^c<2^P.
% \]
% Each prefix has two one-digit extensions, so there are fewer than
% \(2^{P+1}\) words of length \(m\) satisfying~(ii).
% Summing over the \(Q-P\) possible values of \(m\), the total number of
% words satisfying~(ii) is less than
% \[
% 2^{P+1}(Q-P).
% \]

% Finally, the number of words satisfying~(iii) is less than
% \[
% \sum_{c=1}^{P-1}2^c<2^P.
% \]
% Consequently,
% \[
% R
% <
% 2^P+2^{P+1}(Q-P)+2^P
% =
% 2^{P+1}(Q-P+1).
% \]
\end{proof}

\begin{lemma}\label{lemma:sum_of_powers_of_2}
We have
\[
\sum_{w\in\mathcal W}2^{\len(w)}
<
2^{P+Q+3}.
\]
\end{lemma}

\begin{proof}
For the regular words, the words satisfying~(i) contribute at most
\(2^{2P}\).
For each \(n\in\{P+1,\dots,Q\}\), there are fewer than \(2^{P+1}\) words of length \(n\) satisfying~(ii).
Hence
\[
\begin{aligned}
\sum_{w\in\mathcal W_{\mathrm{regular}}}2^{\len(w)}
&\leq
2^{2P}
+
2^{P+1}\sum_{m=P+1}^{Q}2^m
% &=
% 2^{2P}
% +
% 2^{P+1}\left(2^{Q+1}-2^{P+1}\right)\\
<
2^{P+Q+2}.
\end{aligned}
\]
On the other hand, there are fewer than \(2^P\) exceptional words, all of length \(Q\).
Therefore
\[
\sum_{w\in\mathcal W_{\mathrm{exceptional}}}2^{\len(w)}
<
2^{P+Q}.
\]
Combining the two estimates gives
\[
\sum_{w\in\mathcal W}2^{\len(w)}
<
2^{P+Q+2}+2^{P+Q}
<
2^{P+Q+3}.
\]
\end{proof}

\color{black}

The following lemma is immediate:
\begin{lemma}\label{lemma:2_minusL}
We have
\[
\sum_{w\in \mathcal{W}}2^{-\len(w)}=1. 
\]
\end{lemma}
\begin{proof}
This lemma immediately follows from the fact that $C_w$ is an interval of width $2^{-\len(w)}$.
\end{proof}

\subsection{The expected length of $\succ$-LIS}\label{sec:ev_small_block}

We will derive an upper bound for $\lambda_{\succ,1}^{(N)}$ by applying Proposition~\ref{prop:decomp} to $\mathcal{W}=\mathcal{W}(P,Q)$.
Let
\[
x=a_1,a_2,a_3,\dots,a_N\in [0,1),\qquad \text{where }   a_{n+1}=f(a_n)
\]
be the sequence generated by the doubling map $f$.

For a while, we fix an arbitrary $w\in \mathcal{W}$.
%Let $\len(w)$ be the length of $w$.
%Note $P\leq \len(w) \leq Q$.
Also, we fix an interval $J=\{u,u+1,\dots,v-1\}\subset \{1,2,\dots,N\}$.
Let
\(
n:=\# J=v-u
\)
be the size of $J$.

From the subsequence $a_u,a_{u+1},\dots,a_{v-1}$, we extract all terms that are contained in the cylinder $C_w$:
\begin{equation}\label{eq:seq}
\{a_{\tau_1},a_{\tau_2},\dots,a_{\tau_{M_w}}\}
:= 
(a_i:i\in J,a_i\in C_w)=C_w\cap  \{a_u,a_{u+1},\dots,a_{v-1}\}.
\end{equation}
where $M_w$ is the size of the set and
\[
u\leq \tau_1<\tau_2<\dots<\tau_{M_w}<v.
\]
$M_w$ is allowed to be $0$. 

The distribution of the subsequence
$a_{\tau_1},a_{\tau_2},\dots,a_{\tau_{M_w}}$ is \textit{not} independent because these first $\ell(w)$ digits may overlap in the binary expansion of the initial value $x$.
To obtain an independent subsequence, we need to remove some elements from $a_{\tau_1},a_{\tau_2},\dots,a_{\tau_{M_w}}$ that are too close to other element.

Let \(D>\len(w)\) be a positive integer, which will serve as a threshold.
If \(M_w=0\), set
\(M_w'=0\).  If \(M_w>0\), set \(\sigma_1:=\tau_1\).  Having chosen
\(\sigma_j\), let \(\sigma_{j+1}\) be the first \(\tau_i\) satisfying
\[
\tau_i\geq\sigma_j+D,
\]
provided that such a \(\tau_i\) exists.  This produces a subsequence
\begin{equation}\label{eq:subseq}
\{a_{\sigma_1},a_{\sigma_2},\dots,a_{\sigma_{M_w'}}\}
\subset
\{a_{\tau_1},a_{\tau_2},\dots,a_{\tau_{M_w}}\},
\end{equation}
where \(M_w'\leq M_w\) and
\[
u\leq\sigma_1<\sigma_2<\cdots<\sigma_{M_w'}<v,
\qquad
\sigma_{j+1}-\sigma_j\geq D.
\]

Let $L_w$ denote the length of the $\succ$-LIS contained in the sequence $\{a_{\tau_1},a_{\tau_2},\dots,a_{\tau_{M_w}}\}$, and let $W_w$ denote the length of the $\succ$-LIS contained in the sequence $\{a_{\sigma_1},a_{\sigma_2},\dots,a_{\sigma_{M'_w}}\}
$.
Obviously, we have
\begin{equation}\label{eq:L_vs_W}
\begin{array}{ccc}
M_w&\geq &M'_w\\
\rotatebox{90}{$\leq$}&&\rotatebox{90}{$\leq$}\\
L_w&\geq &W_w
\end{array}
\end{equation}

Although $L_w$ is directly related to the evaluation of $\lambda_{\succ,1}^{(N)}$, it is not easy to compute. 
Instead, we will work with $W_w$.
The following inequality is crucial.

\begin{lemma}\label{lemma:combin_W_L}
The random variables $M_w,M_w',W_w,L_w$ satisfy the inequality
\[
L_w\leq W_w+M_w-M_w'.
\]
\end{lemma}

\begin{proof}
Fix one of the $\succ$-LIS contained in $\{a_{\tau_1},a_{\tau_2},\dots,a_{\tau_{M_w}}\}$, whose length is $L_w$.
Whenever one removes some elements from it, the remaining sequence is also an increasing sequence.
Therefore, $\{a_{\sigma_1},a_{\sigma_2},\dots,a_{\sigma_{M'_w}}\}
$ contains an increasing subsequence of length at least $L_w-(M_w-M'_w)$.
% Since the total number of terms removed in the above thinning procedure is $M_w-M_w'$, the length of this remaining sequence is at most $L_w-(M_w-M_w')$.
Hence, we have $W_w\geq L_w-(M_w-M'_w)$, which implies the desired inequality.
\end{proof}

The binary expansion of the initial point \(x\) is written as
\[
x=0.b_1b_2b_3\ldots \big|
\underbrace{
b_u b_{u+1}\ldots
\overbrace{b_{\sigma_1}b_{\sigma_1+1}\ldots}^{\geq D}
\overbrace{b_{\sigma_2}\ldots}^{\geq D}
\overbrace{b_{\sigma_3}\ldots}^{\geq D}
\ldots
b_{v+\len(w)-2}
}_{}
\big|\ldots .
\]
The underbrace indicates the portion of the binary expansion involved
in determining whether or not \(a_j\in C_w\) for \(j\in J\).
For each \(1\leq i\leq M_w'\), define a binary word \(y_i\) of length
\(D\) by
\[
y_i
:=
b_{\sigma_i}b_{\sigma_i+1}\cdots b_{\sigma_i+D-1}.
\]
Thus \(y_i\) is obtained by taking the \(D\) consecutive digits of the
binary expansion of \(x\) starting at the position \(\sigma_i\).
Since \(a_{\sigma_i}\in C_w\), the first \(\len(w)\) digits of \(y_i\)
coincide with \(w\).
Put
\[
t:=D-\len(w),
\]
and denote the last \(t\) digits of \(y_i\) by \(\eta_i\).
We call \(\eta_i\) the \emph{tail} of \(y_i\).

The number \(\sigma_i\) is determined by the digits inspected while
searching for the first occurrence of \(w\) after the preceding
restart time; in particular, this uses digits only up to
\(\sigma_i+\len(w)-1\).
By construction, the search for \(\sigma_{i+1}\) starts only at
\(\sigma_i+D\).
Hence the digits forming \(\eta_i\), namely those at positions
\(\sigma_i+\len(w),\dots,\sigma_i+D-1\), are used neither to determine
\(\sigma_i\) nor to determine any of
\(\sigma_{i+1},\dots,\sigma_{M_w'}\).
The strong Markov property of the Bernoulli digits therefore shows
that, conditional on \(M_w'\) and on the values of
$
\sigma_1,\sigma_2,\dots,\sigma_{M_w'},
$
the tail words
\[
\eta_1,\eta_2,\dots,\eta_{M_w'}
\]
are independent and uniformly distributed on \(\{0,1\}^t\).

Let
\[
\mathrm{Coll}_w:=\{\exists (i,j):(i\neq j)\wedge (y_i=y_j)\}
\]
denote the event that at least two of $y_1,y_2,\dots,y_{M'_w}$ are equal.
As in Lemma~\ref{lemma:birth}, we have
\[
\PP(\mathrm{Coll}_w\,|\,M_w'=k)\leq \binom{k}{2}2^{-t}.
\]
For the same reason as in Proposition~\ref{prop:near_unif}, under $\mathrm{Coll}_w^c$, the distribution of the length of the $\succ$-LIS in 
\(\{a_{\sigma_1},a_{\sigma_2},\dots,a_{\sigma_{M'_w}}\}
\), which is denoted by $W_w$, coincides with the distribution of the length of the LIS contained in a uniform random permutation of length $M_w'$.
Therefore, by Theorem \ref{thm:known}, for any $\epsilon>0$, there exist constants $c(\epsilon)>0$ and $k_0(\epsilon)$ such that
\[
k\geq k_0(\epsilon)\ \Rightarrow\ 
\mathbb{P}\left(W_w>(2+\epsilon)\sqrt{M_w'}\,\middle|\,\mathrm{Coll}_w^c\wedge (M_w'=k)\right)<\exp(-c(\epsilon)\sqrt{k}).
\]
Hence, we have
\begin{equation*}
\begin{aligned}
\mathbb{P}\left(W_w>(2+\epsilon)\sqrt{M_w'}\,\middle|\,M_w'=k\geq k_0(\epsilon)\right)
&\leq \binom{k}{2}2^{-t}+
\exp(-c(\epsilon)\sqrt{k})\\
&\leq \binom{N}{2}2^{-t}+
\exp(-c(\epsilon)\sqrt{k}),
\end{aligned}
\end{equation*}
where the second inequality follows from $M_w'\leq N$.
% Since  we further obtain
% \begin{equation*}
% \begin{aligned}
% \mathbb{P}\left(W_w>(2+\epsilon)\sqrt{M_w'}\,\middle|\,M_w'=k\geq k_0(\epsilon)\right)
% &\leq \binom{N}{2}2^{-t}+
% \exp(-c(\epsilon)\sqrt{k}).
% \end{aligned}
% \end{equation*}
Moreover, since $\exp(-c(\epsilon)\sqrt{k})$ is a monotone decreasing function of $k$, this can be rewritten as
\begin{equation}\label{eq:ev_L_0}
\begin{aligned}
\mathbb{P}\left(W_w>(2+\epsilon)\sqrt{M_w'}\,\middle|\,M_w'\geq k\geq k_0(\epsilon)\right)
&\leq \binom{N}{2}2^{-t}+
\exp(-c(\epsilon)\sqrt{k}).
\end{aligned}
\end{equation}

In what follows, we will use the following inequality obtained by substituting $k=N^{1/10}$ into \eqref{eq:ev_L_0}:
\begin{equation}\label{eq:ev_W_2}
\begin{aligned}
\mathbb{P}\left(W_w>(2+\epsilon)\sqrt{M_w'}\,\middle|\,M_w'\geq N^{1/10}\geq k_0(\epsilon)\right)
&\leq \binom{N}{2}2^{-t}+
\exp\left(-c(\epsilon)N^{1/20}\right).
\end{aligned}
\end{equation}
As $N\to\infty$, the condition $N^{1/10}\geq k_0(\epsilon)$ is automatically satisfied.
In addition, since $M_w'<N^{1/10}$ implies $W_w<N^{1/10}$, we have
\begin{equation}\label{eq:ev_L_2}
\begin{aligned}
\mathbb{P}\left(W_w>(2+\epsilon)\sqrt{M_w'}+N^{1/10}\right)
&\leq 
\binom{N}{2}2^{-t}+
\exp(-c(\epsilon)N^{1/20}).
\end{aligned}
\end{equation}
% From this, we derive the inequality
% \begin{equation}\label{eq:ev_L}
% \begin{aligned}
% \mathbb{P}\left(W_w>(2+\epsilon)\sqrt{M_w'}+N^{1/10}\right)
% &< 2^{-t}N^3\qquad (N\to \infty).
% \end{aligned}
% \end{equation}

\subsection{Estimates of the random variables $M_w$ and $M_w-M'_w$}

In this subsection, we will give estimates of the two random variables $M_w$ and $M_w-M'_w$.
For this, we use the following 
%\emph{Bernstein inequality}
Bernstein inequality for independent (but not necessarily identically distributed) Bernoulli variables:
\begin{lemma}
%[Bernstein bound for Bernoulli variables]
\label{lem:CB-extract}
Let $\xi_1,\dots,\xi_n$ be independent Bernoulli random variables.
%\marginpar{\tiny YN: ここで略記を色々導入していますが、後で使わないなら消してもいいかも？}
% Then for any $u>0$, we have
% \begin{equation}\label{eq:CB-upper}
% \PP(X-\mu\ge u)
% \le
% \exp\!\left(-\frac{u^2}{2(\mu+u/3)}\right)
% \le
% \exp\!\left(
% -\frac{3}{8}\min\!\left\{\frac{u^2}{\mu},\,u\right\}
% \right).
% \end{equation}
Then, for any $0<\delta<1$ and $s\ge 0$, we have
\begin{equation}\label{eq:CB-shifted}
\PP\left(\sum_{i=1}^n \xi_i>s+ (1+\delta)\sum_{i=1}^n\PP(\xi_i=1) \right)\le \exp(-c_\delta s),
\end{equation}
where $c_\delta:=\frac{3\delta}{2(3+\delta)}$.
\end{lemma}
\begin{proof}
For independent zero-mean random variables $X_1,X_2,\cdots,X_n$ such that $|X_i|\le M$ almost surely for all $i$,
the classical Bernstein's inequality gives
\[
\PP\left(\sum_{i=1}^nX_i\ge u\right)
\le
\exp\left(
-\frac{\frac12u^2}{\sum_{i=1}^n \EE[X_i^2]+\frac13uM}
\right)
\]
for every \(u>0\).
See, for example,
\cite[Section~2.7, Eq.~(2.10)]{BoucheronLugosiMassart2013}.
Set 
\[
p_i:=\PP(\xi_i=1),\qquad \mu:=\sum_{i=1}^n p_i.
\]
Since \(\xi_i-p_i\le 1\) and
$
\sum_{i=1}^n \EE[(\xi_i-p_i)^2]
=
\sum_{i=1}^n p_i(1-p_i)
\le
%\sum_{i=1}^n p_i=
\mu,
$
by the Bernstein's inequality, for every \(u>0\),
\[
\PP\left(\sum_{i=1}^n(\xi_i-p_i)\ge u\right)
\le
\exp\left(
-\frac{u^2}{2(\mu+u/3)}
\right).
\]
Taking \(u=s+\delta\mu\), we obtain
\[
\PP\left(\sum_{i=1}^n\xi_i>s+(1+\delta)\mu\right)
\le
\exp\left(
-\frac{(s+\delta\mu)^2}
{2(\mu+(s+\delta\mu)/3)}
\right).
\]
It remains only to check
\[
\frac{(s+\delta\mu)^2}
{2(\mu+(s+\delta\mu)/3)}
\ge
\frac{3\delta}{2(3+\delta)}s,
\]
which is followed by clearing denominators using
\[
(3+\delta)(s+\delta\mu)^2
-
\delta s\bigl((3+\delta)\mu+s\bigr)
=
(3+\delta)\delta^2\mu^2
+
(3+\delta)\delta\mu s
+
3s^2
\ge 0.
\]
Thus the desired estimate follows.
%{\color{red} See [] for example.}
%We will give a proof of this lemma in the Appendix. 
\end{proof}

We write $\mathrm{hit}(i)$ for the event that the word $w$ coincides with $b_i b_{i+1}\ldots b_{i+\len(w)-1}$.
Define a random variable $I_i$ by
\[
I_i=
\mathbf{1}_{\mathrm{hit}(i)}=
\begin{cases}
1 & (\text{the word $w$ coincides with } b_i b_{i+1}\ldots b_{i+\len(w)-1}),\\
0 & (\text{otherwise}).
\end{cases}
\]
Since $M_w=\sum_{i=u}^{v-1} I_i$, the expectation value of $M_w$ is
\begin{equation}\label{eq:exp_M}
\mathbb{E}[M_w]
=
\sum_{i=u}^{v-1}\mathbb{E}[I_i]
=
\frac{n}{2^{\len(w)}}.
\end{equation}

\begin{prop}\label{prop:ev_M}
For $0<\delta<1$, there exists a real positive number $c_\delta>0$ such that, for any $f>0$,
\begin{equation}\label{eq:ev_M}
\PP\left(M_w>(1+\delta)\frac{n}{2^{\len(w)}}+f\right)<\ell(w)\exp\left(-c_\delta\frac{f}{\len(w)}\right).
\end{equation}
\end{prop}

\begin{proof}
We decompose $M_w$ as
\[
\begin{aligned}
M_w&=(I_u+I_{u+\len(w)}+I_{u+2\len(w)}+\dots)+(I_{u+1}+I_{u+\len(w)+1}+I_{u+2\len(w)+1}+\dots)\\
&\quad +\cdots+(I_{u+\len(w)-1}+I_{u+2\len(w)-1}+\cdots),
\end{aligned}
\]
where each 
\[
I_{u+i}+I_{u+\len(w)+i}+I_{u+2\len(w)+i}+\dots=\sum_{k=0}^{\lfloor (n-i)/\len(w) \rfloor} I_{u+i+k\len(w)}
\]
consists of mutually independent random variables.
Let $c_\delta$ be the real positive number defined in Lemma \ref{lem:CB-extract}.
Then, by the concentration inequality \eqref{eq:CB-shifted}, we have
\[
\PP\left(
\sum_{k}I_{u+i+k\len(w)} 
>(1+\delta)\sum_k\EE[I_{u+i+k\len(w)}] +\frac{f}{\len(w)}\right)<\exp\left(-c_\delta\frac{f}{\len(w)}\right).
\]
By summing up these inequalities for $i=0,1,\dots,\len(w)-1$, we obtain the desired result.
%from Markov's inequality.
\end{proof}

Next, we consider $M_w-M_w'$, which is the number of elements of the difference set 
\[
\{a_{\tau_1},a_{\tau_2},\dots,a_{\tau_{M_w}}\}
\setminus
\{a_{\sigma_1},a_{\sigma_2},\dots,a_{\sigma_{M'_w}}\}
.
\]
% The thinning condition was that the difference between the indices of two terms be less than $D$.
% Accordingly, 
Define the random variable $V_i$, which means ``$a_i$ is contained in $\{a_{\tau_1},a_{\tau_2},\dots,a_{\tau_{M_w}}\}$, but has a possibility not to be contained in $\{a_{\sigma_1},a_{\sigma_2},\dots,a_{\sigma_{M'_w}}\}
$,'' by
\begin{equation}\label{eq:def_of_V}
V_i:=
\begin{cases}
1 & \mathrm{hit}(i)\wedge (\exists d\in \{1,2,\dots,D-1\}\ \ \text{s.t.}\ \ \mathrm{hit}(i+d)),\\
0 & (\text{otherwise}).
\end{cases}    
\end{equation}
Then, we have the inequality
\begin{equation}\label{eq:ev_M-M}
M_w-M_w'\leq \sum_{i=u}^{v-1}V_i.
\end{equation}

From \eqref{eq:def_of_V}, we have
\begin{equation}\label{eq:ev_of_V}
\EE[V_i]
\leq
\sum_{d=1}^{D-1}
\PP\bigl(\mathrm{hit}(i)\wedge \mathrm{hit}(i+d)\bigr).
\end{equation}
Let \(C\leq \len(w)\) denote the period of \(w\).
For \(1\leq d<C\), the two events $\mathrm{hit}(i)$ and $\mathrm{hit}(i+d)$ cannot occur simultaneously:
\[
\PP\bigl(\mathrm{hit}(i)\wedge \mathrm{hit}(i+d)\bigr)=0.
\]
For \(C\leq d<\len(w)\), if the two events $\mathrm{hit}(i)$ and $\mathrm{hit}(i+d)$ occur simultaneously, they prescribe the digits on the union
\[
\{i,i+1,\dots,i+\len(w)-1\}
\cup
\{i+d,i+d+1,\dots,i+d+\len(w)-1\},
\]
which contains \(\len(w)+d\) numbers.
Therefore, we have
\[
\PP\bigl(\mathrm{hit}(i)\wedge \mathrm{hit}(i+d)\bigr)
\leq
2^{-(\len(w)+d)}.
\]
For \(\ell(w)\leq d<D\), the two events involve disjoint sets of \(\ell(w)\) digits, and hence
\[
\PP\bigl(\mathrm{hit}(i)\wedge \mathrm{hit}(i+d)\bigr)
=
2^{-2\len(w)}.
\]

Consequently, from \eqref{eq:ev_of_V}, we have
\[
\begin{aligned}
\EE[V_i]
&\leq
\sum_{d=C}^{\len(w)-1}2^{-(\len(w)+d)}
+
\sum_{d=\len(w)}^{D-1}2^{-2\len(w)}<
\frac{2}{2^{\len(w)+C}}
+
\frac{D-\len(w)}{4^{\len(w)}}\\
&=
\frac{2}{2^{\len(w)+C}}
+
\frac{t}{4^{\len(w)}}
\leq
\frac{t+2}{2^{\len(w)+C}}.
\end{aligned}
\]
The first sum is understood to be zero when \(C=\len(w)\).

It now follows from \eqref{eq:ev_M-M} that
\begin{equation}\label{eq:upper_MM}
\EE[M_w-M_w']
\leq
\sum_{i=u}^{v-1}\EE[V_i]
\leq
\frac{(t+2)n}{2^{\len(w)+C}}.
\end{equation}
\color{black}

\begin{lemma}\label{lemma:genMM}
For $0<\delta<1$, there exists a real positive number $c_\delta>0$ such that, for any $g>0$,
\[
\mathbb{P}\left(M_w-M_w'>(1+\delta)\frac{(t+2) n}{2^{\len(w)+C}}+g
\right)<
(D+\len(w))
\exp
\left(-c_\delta\frac{g}{D+\len(w)}\right).
\]
\end{lemma}

\begin{proof}
The random variable $\sum_{i=u}^{v-1}V_i$ can be decomposed as
\[
\sum_{i=u}^{v-1}V_i
=
(V_u+V_{u+D+\len(w)}+V_{u+2(D+\len(w))}+\dots)+
(V_{u+1}+V_{u+D+\len(w)+1}+V_{u+2(D+\len(w))+1}+\dots)+\cdots,
\]
where 
each 
\[
V_{u+i}+V_{u+(D+\len(w))+i}+V_{u+2(D+\len(w))+i}+\dots
=\sum_{k=0}^{\lfloor (n-i)/(D+\len(w)) \rfloor} V_{u+i+k(D+\len(w))}
\]
consists of mutually independent random variables.
A similar argument as in the proof of Proposition \ref{prop:ev_M} leads to the desired inequality.
\end{proof}

The following two lemmas immediately follow from Lemma \ref{lemma:genMM}.

\begin{lemma}\label{lemma:MM}
If $w\in \mathcal{W}_{\mathrm{regular}}$, we have
\[
\mathbb{P}\left(M_w-M_w'>(1+\delta)\frac{(t+2) n}{4^{P}}+g
\right)<
(2Q+t)
\exp
\left(-c_\delta\frac{g}{2Q+t}\right).
\]
\end{lemma}

\begin{proof}
Since $D=\len(w)+t$, we have $D+\len(w)=2\len(w)+t
\leq 2Q+t$.
If $w\in \mathcal{W}_{\mathrm{regular}}$, it follows that $P\leq C\leq \len(w)$, which implies the claim.
\end{proof}

\begin{lemma}\label{lemma:MM_ex}
If $w\in \mathcal{W}_{\mathrm{exceptional}}$, we have
\[
\mathbb{P}\left(M_w-M_w'>(1+\delta)\frac{(t+2) n}{2^{Q+1}}+g
\right)<
(2Q+t)
\exp
\left(-c_\delta\frac{g}{2Q+t}\right).
\]
\end{lemma}
\begin{proof}
If $w\in \mathcal{W}_{\mathrm{exceptional}}$, we have $\len(w)=Q$, which implies the claim.
\end{proof}

% \subsection{$\mathcal{W}_{\mathrm{regular}}$ vs.\ $\mathcal{W}_{\mathrm{exceptional}}$}

% The basic idea of proofs in this paper is to find a subsequence distributed like a uniform random permutation.
% %by ``thinning out'' part of the sequence $I_w\cap\{a_i:i\in J\}$.
% For $w\in \mathcal{W}_{\mathrm{regular}}$, in order for the number of points removed in this thinning procedure (which is $\sim O(n/4^P)$; see Lemma~\ref{lemma:MM}) to be small, it suffices to take the period length of $w$ (that is, the size of $P$) large.
% On the other hand, for $w\in \mathcal{W}_{\mathrm{exceptional}}$, we ensure the smallness of the number of removed points (which is $\sim O(n/2^{Q+1})$; see Lemma~\ref{lemma:MM_ex}) by taking $Q$ large.

In what follows, we assume that
\begin{equation}\label{eq:Q_vs_P}
Q\geq 2P.
\end{equation}
Under this assumption, Lemma~\ref{lemma:MM} and Lemma~\ref{lemma:MM_ex} can be combined as follows:
\begin{prop}\label{prop:M-M'}
Under \eqref{eq:Q_vs_P}, we have
\[
\mathbb{P}\left(M_w-M_w'>(1+\delta)\frac{(t+2) n}{4^{P}}+g
\right)<
(2Q+t)
\exp
\left(-c_\delta\frac{g}{2Q+t}\right)
\]
 for any $w\in \mathcal{W}$.
\end{prop}

\subsection{Estimate of $\lambda_{\succ,1}^{(N)}$}

In this final subsection, we derive an upper bound for $\EE[\lambda_{\succ,1}^{(N)}]$ by applying Proposition~\ref{prop:decomp}.
For this, we sum all the concentration inequalities obtained in the previous subsections over all pairs consisting of $w\in\mathcal{W}$ and an interval $J=[u,v)$.
In what follows, we rewrite $W_w,L_w,M_w,M_w'$ in the previous subsection as $W_w(J),L_w(J),M_w(J),M_w'(J)$ to clarify the dependency of the choice of $J$.

\begin{prop}\label{prop:ideal}
Let $f,g$ be arbitrary positive parameters.
Assume \eqref{eq:Q_vs_P}.
For sufficiently large $N$, define the function $\mathcal{H}$ by
\[
\mathcal{H}:=RN^2\left(
\binom{N}{2}2^{-t}+
\exp(-c(\epsilon)N^{1/20})
%2^{-t}N^3+
+Q\exp\left(-c_\delta\frac{f}{Q}\right)
+(2Q+t)
\exp
\left(-c_\delta\frac{g}{2Q+t}\right)
\right).
\]
Then, the following events~(i)--(iii) occur simultaneously with probability at least $1-\mathcal{H}$.
\begin{enumerate}
\item[(i)] For every $w\in \mathcal{W}$ and every interval $J\subset [1,N]$,
\[
W_w(J)\leq (2+\epsilon)\sqrt{M'_w(J)}+N^{1/10}
\]
holds.
\item[(ii)] For every $w\in \mathcal{W}$ and every interval $J\subset [1,N]$,
\[
M_w(J)\leq (1+\delta)\frac{\# J}{2^{\len(w)}}+f
\]
holds.
\item[(iii)] For every $w\in \mathcal{W}$ and every interval $J\subset [1,N]$,
\[
M_w(J)-M_w'(J)\leq (1+\delta)\frac{(t+2)\# J}{4^P}+g
\]
holds.
\end{enumerate}
\end{prop}

\begin{proof}
The estimates for (i)--(iii) come respectively from \eqref{eq:ev_L_2}, Proposition~\ref{prop:ev_M}, and Proposition~\ref{prop:M-M'}.
The function $\mathcal{H}$ is obtained by summing the corresponding right-hand sides. 
Note that the total number of pairs
\[
(w,J)\in \mathcal{W}\times \{J=[u,v):1\leq u<v\leq N+1\}
\]
is at most $RN^2$.
\end{proof}

% We will denote by
% \[
% \mathrm{Ideal}
% \]
% the event on which all of (i)--(iii) in Proposition~\ref{prop:ideal} occur.

Let
\[
(1=t^\ast_0\leq t^\ast_1\leq \dots\leq t^\ast_{R}=N+1)\in \Delta_N^{(R)}
\]
be a partition that gives the maximum on the right-hand side of Proposition~\ref{prop:decomp}.
Arrange the elements of $\mathcal{W}=\mathcal{W}(P,Q)$ as
\[
w_0,w_1,\dots,w_{R-1}
\]
such that $C_{w_0}\prec C_{w_1}\prec \dots \prec C_{w_{R-1}}$.
%Set $I_r:=C_{w_r}$.

For each $w=w_r\in \mathcal{W}$, let $J_{w}^\ast=[t^\ast_r,t^\ast_{r+1})$ be the $r$-th interval.
Below, we abbreviate
\[
L^\ast_w:=L_w(J_w^\ast),\qquad
W^\ast_w:=W_w(J_w^\ast),\qquad
M^\ast_w:=M_w(J_w^\ast),\qquad
{M'_w}^\ast:=M'_w(J_w^\ast).
\]
Then, we have the identity
\begin{equation}\label{eq:decomp_lambda_1}
\lambda_{\succ,1}^{(N)}=\sum_{w\in \mathcal{W}}L^\ast_w
\end{equation}

%\subsubsection{Estimate of $L_w$}

Under the event on which all of (i)--(iii) in Proposition~\ref{prop:ideal} occur, we have
\begin{align*}
\lambda_{\succ,1}^{(N)}
&=\sum_{w\in \mathcal{W}}L^\ast_w\\
&\leq 
\sum_{w\in \mathcal{W}}(W^\ast_w+M^\ast_w-{M'_w}^\ast)\qquad (\text{Lemma \ref{lemma:combin_W_L}})\\
&\leq \sum_{w\in \mathcal{W}}\left((2+\epsilon)\sqrt{{M_w'}^\ast}+N^{1/10}+(1+\delta)\frac{(t+2)\# J^\ast_w}{4^P}+g\right)
\qquad (\text{Prop.~\ref{prop:ideal} (i),(iii)})
\\
% &\leq \sum_{w\in \mathcal{W}}\left((2+\epsilon)\sqrt{{M_w'}^\ast}+N^{1/10}+(1+\delta)\frac{(t+2)\# J^\ast_w}{4^P}+g\right)\\
&=(2+\epsilon)\sum_{w\in \mathcal{W}} \sqrt{{M_w'}^\ast}+(1+\delta)\frac{(t+2)N}{4^P}+R(g+N^{1/10}).
\end{align*}

Here, by the Cauchy--Schwarz inequality, we obtain
\begin{align*}
\sum_{w\in \mathcal{W}}\sqrt{{M_w'}^\ast} 
&\leq
\sqrt{
\left(\sum_{w\in \mathcal{W}} \frac{1}{2^{\len(w)}}\right)
\left(\sum_{w\in \mathcal{W}} 2^{\ell(w)}{M_w'}^\ast\right)
}\\
&=
\sqrt{
\sum_{w\in \mathcal{W}} 2^{\len(w)}{M_w'}^\ast}
\qquad (\text{Lemma \ref{lemma:2_minusL}})
\\
&\leq 
\sqrt{
\sum_{w\in \mathcal{W}} 2^{\len(w)}M^\ast_w}\\
&\leq 
\sqrt{(1+\delta)N+(\sum_{w\in \mathcal{W}} 2^{\len(w)})f}
\qquad(\text{Prop.~\ref{prop:ideal} (ii)})
\\
&\leq \sqrt{(1+\delta)N+2^{P+Q+3}f}
\qquad (\text{Lemma \ref{lemma:sum_of_powers_of_2}}).
\end{align*}
%\marginpar{\tiny WとWregの違い問題 $\to$ Lemma \ref{lemma:sum_of_powers_of_2}を変更, 合わせてここで二箇所$+1$を$+3$に変更しました. 次ページの$R<2^{P+1}(Q-P+1)<N^{1/3}$も}
From these two inequalities, it follows that
\[
\lambda_{\succ,1}^{(N)}\leq \mathcal{F},
\]
where 
\begin{equation}\label{eq:regularW}
\mathcal{F}:=
(2+\epsilon)\sqrt{(1+\delta)N+2^{P+Q+3}f}+(1+\delta)\frac{(t+2)N}{4^P}+R(g+N^{1/10}).
\end{equation}

% From this point, define the right-hand side of \eqref{eq:regularW} by
% \[
% \begin{aligned}
% \mathcal{F}&:=
% (2+\epsilon)\sqrt{(1+\delta)N+2^{P+Q+1}f}+(1+\delta)\frac{(t+2)N}{4^P}+R(g+N^{1/10}).
% \end{aligned}
% \]
% Then, under the event $\mathrm{Ideal}$, we see that
% \[
% \lambda_{\succ,1}^{(N)}\leq \mathcal{F}
% \]
% holds.

In other words, we obtain the following proposition.

\begin{prop}\label{prop:final_ev}
Under the inequality \eqref{eq:Q_vs_P}, we have
\[
\PP\left(\lambda_{\succ,1}^{(N)}>\mathcal{F}\right)<\mathcal{H}.
\]
\end{prop}

% \subsubsection{Estimate of $\EE[\lambda_{\succ,1}^{(N)}]$}

Finally, we complete the estimation of $\EE[\lambda_{\succ,1}^{(N)}]$.
For a sufficiently small constant $\theta>0$,
we choose the parameters as
\[
\begin{gathered}
\delta=\frac{\epsilon}{10},\quad P\sim \left(\frac{1}{4}+\theta\right)\log_2N,\quad
Q\sim \left(\frac{3}{4}-2\theta\right)\log_2N,\\
 t\sim 10\log_2N,\quad 
f\sim N^{\theta/2},\quad
 g\sim N^{1/10}
\end{gathered}
\]
as $N\to \infty$.
Then the inequality \eqref{eq:Q_vs_P} is satisfied, so Proposition~\ref{prop:final_ev} applies.

By Proposition \ref{prop:ev_R}, we have
\[
R<2^{P+1}(Q-P+1)<N^{1/3}\qquad (N\to \infty).
\]
From this, we see that each term appearing in $\mathcal{F}$ is of order at most $N^{1/2}$, and that its leading term is
\[
\mathcal{F}\sim (2+\epsilon)\sqrt{(1+\epsilon/10)N}\qquad (N\to \infty).
\]
On the other hand, since the leading term of $\mathcal{H}$ is 
\[
RN^2\binom{N}{2}2^{-t}, 
\]
it follows that
\[
\mathcal{H}=o(N^{-1}).
\]

As $\lambda_{\succ,1}^{(N)}\leq N$, Proposition \ref{prop:final_ev} implies that
\[
\EE[\lambda_{\succ,1}^{(N)}]\leq (2+\epsilon)\sqrt{(1+\epsilon/10)N}+N\cdot o(N^{-1}).
\]
Therefore,
\[
\EE[\lambda_{\succ,1}^{(N)}]\leq \left(2+\frac{6}{5}\epsilon\right)\sqrt{N}\qquad (N\to \infty)
\]
follows.
Finally, by replacing $\frac{6}{5}\epsilon$ with $\epsilon$, we obtain the desired inequality.

\appendix

\section*{Acknowledgements}
S.~Iwao was supported
by JSPS KAKENHI Grant Numbers 22K03239 and 23K03056.
F.~Nakamura was supported
by JSPS KAKENHI Grant Numbers 23K25785 and 24K16942.
Y.~Nakano was supported by JSPS KAKENHI Grant  Number 23K03188 and JST PRESTO Grant Number JPMJPR25K8.

\section*{Declaration of Generative AI Usage}

In this paper, the authors utilized AI to generate the main ideas for
our proofs, refine the language, and check for errors. Although the
AI-generated proofs were not flawless, they provided the authors with
the inspiration and courage to revisit a problem that had been left
unsolved for a long time. This paper is the result of human
mathematicians creatively reconstructing and expanding upon the ideas
initially provided by the AI. The authors bear full responsibility for
the accuracy and integrity of the content.

\section*{Data availability statement}

No new data was generated or analyzed during this study.

\section*{Declaration of interest}

The authors declare no conflicts of interest associated with this manuscript.

\end{document}

%% file: fig_partition.tex
%% --- Figure: Proposition partition (horizontal layout, ellipsis suggesting many R) ---
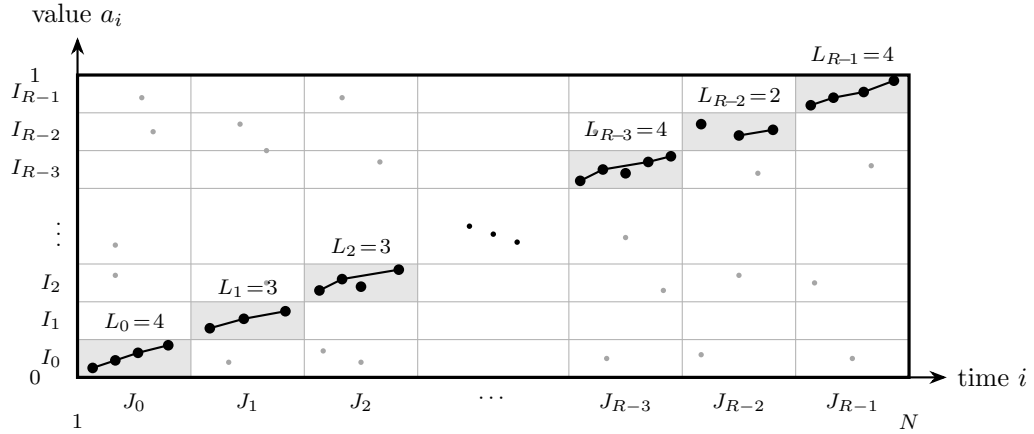
\begin{figure}[htbp]
\centering
\begin{tikzpicture}[>=Stealth, scale=0.5, every node/.style={font=\small}]
  %% Layout (TikZ units): 7 columns and 7 rows including a wider middle ellipsis
  %% Column boundaries (x): 0, 3, 6, 9, 13, 16, 19, 22  (cols 0..2 left, ellipsis, cols R-3..R-1 right)
  %% Row boundaries (y):    0, 1, 2, 3, 5,  6,  7,  8

  %% ===== Diagonal blocks fill =====
  \fill[black!10] (0, 0)  rectangle (3, 1);    % r=0
  \fill[black!10] (3, 1)  rectangle (6, 2);    % r=1
  \fill[black!10] (6, 2)  rectangle (9, 3);    % r=2
  \fill[black!10] (13, 5) rectangle (16, 6);   % r=R-3
  \fill[black!10] (16, 6) rectangle (19, 7);   % r=R-2
  \fill[black!10] (19, 7) rectangle (22, 8);   % r=R-1

  %% ===== Grid lines (visible part) =====
  \foreach \x in {0, 3, 6, 9, 13, 16, 19, 22} {
    \draw[thin, black!30] (\x, 0) -- (\x, 8);
  }
  \foreach \y in {0, 1, 2, 3, 5, 6, 7, 8} {
    \draw[thin, black!30] (0, \y) -- (22, \y);
  }
  \draw[very thick] (0, 0) rectangle (22, 8);

  %% ===== Diagonal continuation indicator (in the central ellipsis area) =====
  \node[font=\Huge, fill=white, inner sep=2pt] at (11, 4) {$\ddots$};

  %% ===== Axes =====
  \draw[->, thick] (0, 0) -- (23.0, 0) node[right] {time $i$};
  \draw[->, thick] (0, 0) -- (0, 9.0)  node[above] {value $a_i$};

  %% ===== Block / interval labels =====
  \node[below] at (1.5,  -0.15) {\scriptsize $J_0$};
  \node[below] at (4.5,  -0.15) {\scriptsize $J_1$};
  \node[below] at (7.5,  -0.15) {\scriptsize $J_2$};
  \node[below] at (11,   -0.15) {\scriptsize $\cdots$};
  \node[below] at (14.5, -0.15) {\scriptsize $J_{R-3}$};
  \node[below] at (17.5, -0.15) {\scriptsize $J_{R-2}$};
  \node[below] at (20.5, -0.15) {\scriptsize $J_{R-1}$};

  \node[left] at (-0.15, 0.5) {\scriptsize $I_0$};
  \node[left] at (-0.15, 1.5) {\scriptsize $I_1$};
  \node[left] at (-0.15, 2.5) {\scriptsize $I_2$};
  \node[left] at (-0.15, 4)   {\scriptsize $\vdots$};
  \node[left] at (-0.15, 5.5) {\scriptsize $I_{R-3}$};
  \node[left] at (-0.15, 6.5) {\scriptsize $I_{R-2}$};
  \node[left] at (-0.15, 7.5) {\scriptsize $I_{R-1}$};

  %% ===== Axis tick labels =====
  \node[below] at (0,  -0.7) {\scriptsize $1$};
  \node[below] at (22, -0.7) {\scriptsize $N$};
  \node[left]  at (-0.7, 0)  {\scriptsize $0$};
  \node[left]  at (-0.7, 8)  {\scriptsize $1$};

  %% ===== Black dots: varied counts per block (illustrate L_r values 2, 3, 4) =====
  % r=0: 4 dots, all increasing → L_0 = 4
  \fill (0.4, 0.25) circle (4pt); \fill (1.0, 0.45) circle (4pt);
  \fill (1.6, 0.65) circle (4pt); \fill (2.4, 0.85) circle (4pt);
  % r=1: 3 dots, increasing → L_1 = 3
  \fill (3.5, 1.3) circle (4pt); \fill (4.4, 1.55) circle (4pt); \fill (5.5, 1.75) circle (4pt);
  % r=2: 4 dots with one out-of-order → L_2 = 3
  \fill (6.4, 2.3) circle (4pt); \fill (7.0, 2.6) circle (4pt);
  \fill (7.5, 2.4) circle (4pt); \fill (8.5, 2.85) circle (4pt);
  % r=R-3: 5 dots with one out-of-order → L_{R-3} = 4
  \fill (13.3, 5.2) circle (4pt); \fill (13.9, 5.5) circle (4pt);
  \fill (14.5, 5.4) circle (4pt); \fill (15.1, 5.7) circle (4pt); \fill (15.7, 5.85) circle (4pt);
  % r=R-2: 3 dots, mostly decreasing → L_{R-2} = 2
  \fill (16.5, 6.7) circle (4pt); \fill (17.5, 6.4) circle (4pt); \fill (18.4, 6.55) circle (4pt);
  % r=R-1: 4 dots, all increasing → L_{R-1} = 4
  \fill (19.4, 7.2) circle (4pt); \fill (20.0, 7.4) circle (4pt);
  \fill (20.8, 7.55) circle (4pt); \fill (21.6, 7.85) circle (4pt);

  %% ===== Off-diagonal sparse gray dots =====
  \foreach \x/\y in {
    1.0/3.5, 2.0/6.5, 1.7/7.4,
    4.0/0.4, 5.0/2.5, 4.3/6.7,
    6.5/0.7, 8.0/5.7, 7.0/7.4,
    14.0/0.5, 15.5/2.3, 14.5/3.7, 13.7/6.5,
    16.5/0.6, 17.5/2.7, 18.0/5.4,
    19.5/2.5, 20.5/0.5, 21.0/5.6,
    1.0/2.7, 5.0/6.0, 7.5/0.4
  } { \fill[black!35] (\x, \y) circle (2pt); }

  %% ===== L_r connectors (LIS path within each visible diagonal block) =====
  \draw[thick] (0.4, 0.25) -- (1.0, 0.45) -- (1.6, 0.65) -- (2.4, 0.85);    % L_0=4
  \draw[thick] (3.5, 1.3)  -- (4.4, 1.55) -- (5.5, 1.75);                   % L_1=3
  \draw[thick] (6.4, 2.3)  -- (7.0, 2.6)  -- (8.5, 2.85);                   % L_2=3, skips (7.5, 2.4)
  \draw[thick] (13.3, 5.2) -- (13.9, 5.5) -- (15.1, 5.7) -- (15.7, 5.85);   % L_{R-3}=4, skips (14.5, 5.4)
  \draw[thick] (17.5, 6.4) -- (18.4, 6.55);                                 % L_{R-2}=2
  \draw[thick] (19.4, 7.2) -- (20.0, 7.4) -- (20.8, 7.55) -- (21.6, 7.85);  % L_{R-1}=4

  %% ===== L_r labels (above the path inside each block) =====
  \node[anchor=south, font=\scriptsize] at (1.5,  0.95) {$L_0\!=\!4$};
  \node[anchor=south, font=\scriptsize] at (4.5,  1.85) {$L_1\!=\!3$};
  \node[anchor=south, font=\scriptsize] at (7.5,  2.95) {$L_2\!=\!3$};
  \node[anchor=south, font=\scriptsize] at (14.5, 5.95) {$L_{R\!-\!3}\!=\!4$};
  \node[anchor=south, font=\scriptsize] at (17.5, 6.85) {$L_{R\!-\!2}\!=\!2$};
  \node[anchor=south, font=\scriptsize] at (20.5, 7.95) {$L_{R\!-\!1}\!=\!4$};

\end{tikzpicture}
\caption{Mechanism of the partition from Proposition~\ref{prop:decomp}. The horizontal axis is the time index $i$ (typically very long, since $N$ is large) and the vertical axis is the value $a_i \in [0,1)$. The index range is split into $R$ equal blocks $J_0, \ldots, J_{R-1}$ and the value range into $R$ equal $\succ$-intervals $I_0, \ldots, I_{R-1}$; the $R$ diagonal blocks $J_r \times I_r$ (shaded) are the ones we keep. Black dots mark positions $i$ with $a_i \in I_r$ inside the diagonal block, gray dots mark off-diagonal positions and are discarded. The solid path inside each diagonal block is a $\succ$-LIS of length $L_r$.}
\label{fig:partition}
\end{figure}

%% file: fig_word_partition.tex
%% --- Figure: word partition for P = 2, Q = 4 ---
\begin{figure}[htbp]
\centering
\begin{tikzpicture}[scale=0.65, every node/.style={font=\footnotesize}]

  %% ===== [0, 1) partition bar =====
  \fill[black!28] (0, 0) rectangle (1, 1);
  \fill[black!5]  (1, 0) rectangle (2, 1);
  \fill[black!5] (2, 0) rectangle (4, 1);
  \fill[black!5]  (4, 0) rectangle (8, 1);
  \fill[black!5] (8, 0) rectangle (12, 1);
  \fill[black!5]  (12, 0) rectangle (14, 1);
  \fill[black!5] (14, 0) rectangle (15, 1);
  \fill[black!28] (15, 0) rectangle (16, 1);

  \draw[very thick] (0, 0) rectangle (16, 1);
  \foreach \x in {1, 2, 4, 8, 12, 14, 15} {
    \draw[thick] (\x, 0) -- (\x, 1);
  }

  \node[font=\scriptsize] at (0.5, 0.5)  {$0000$};
  \node[font=\scriptsize] at (1.5, 0.5)  {$0001$};
  \node[font=\scriptsize] at (3, 0.5)    {$001$};
  \node[font=\scriptsize] at (6, 0.5)    {$01$};
  \node[font=\scriptsize] at (10, 0.5)   {$10$};
  \node[font=\scriptsize] at (13, 0.5)   {$110$};
  \node[font=\scriptsize] at (14.5, 0.5) {$1110$};
  \node[font=\scriptsize] at (15.5, 0.5) {$1111$};

  \foreach \x/\l in {0/$0$, 1/{$\tfrac{1}{16}$}, 2/{$\tfrac{1}{8}$}, 4/{$\tfrac{1}{4}$}, 8/{$\tfrac{1}{2}$}, 12/{$\tfrac{3}{4}$}, 14/{$\tfrac{7}{8}$}, 15/{$\tfrac{15}{16}$}, 16/$1$} {
    \node[font=\tiny, below] at (\x, -0.05) {\l};
  }

  %% ===== Width annotations =====
  \draw[<->, thin] (4, -0.7) -- (8, -0.7) node[midway, below, font=\tiny] {$|C_{01}| = 2^{-2}$};
  \draw[<->, thin] (2, -0.7) -- (4, -0.7) node[midway, below, font=\tiny] {$2^{-3}$};
  \draw[<->, thin] (0, -0.7) -- (1, -0.7) node[midway, below, font=\tiny] {$2^{-4}$};

\end{tikzpicture}
\caption{Partition of $[0,1)$ by words $w \in {\mathcal{W}}$ for $P = 2$, $Q = 4$. The interval $[0,1)$ is split into 8 intervals $C_w$, with the width of $C_w$ equal to $2^{-\len(w)}$. Light cells are words in ${\mathcal W}_{\text{regular}}$; the two dark cells $C_{0000}$ and $C_{1111}$ correspond to $\mathcal{W}_{\text{exceptional}}$, the words forced to terminate at length $Q = 4$ because their period-1 pattern persisted.}
\label{fig:word_partition}
\end{figure}
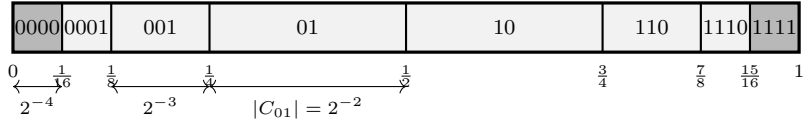

%% file: fig_word_partition2.tex
%% --- Figure: P = 3, Q = 6 partition ---
\begin{figure}[htbp]
\centering
\begin{tikzpicture}[>=Stealth, scale=0.85, every node/.style={font=\footnotesize}]
  \pgfmathsetmacro{\sc}{14/64}

  %% ===== Full [0, 1) bar =====
  \foreach \s/\e/\tp in {
    0/1/1, 1/2/0, 2/4/0, 4/8/0, 8/16/0,
    16/20/0, 20/21/0, 21/22/1, 22/24/0, 24/32/0,
    32/40/0, 40/42/0, 42/43/1, 43/44/0, 44/48/0,
    48/56/0, 56/60/0, 60/62/0, 62/63/0, 63/64/1
  } {
    \pgfmathsetmacro{\xl}{\s * \sc}
    \pgfmathsetmacro{\xr}{\e * \sc}
    \ifnum\tp=1
      \fill[black!30] (\xl, 0) rectangle (\xr, 1);
    \else
      \fill[black!7]  (\xl, 0) rectangle (\xr, 1);
    \fi
  }
  \draw[very thick] (0, 0) rectangle ({64*\sc}, 1);
  \foreach \x in {1,2,4,8,16,20,21,22,24,32,40,42,43,44,48,56,60,62,63} {
    \draw[thin] ({\x*\sc}, 0) -- ({\x*\sc}, 1);
  }

  %% In-bar labels (wide enough cells)
  \node at ({12*\sc}, 0.5) {\scriptsize $001$};
  \node at ({28*\sc}, 0.5) {\scriptsize $011$};
  \node at ({36*\sc}, 0.5) {\scriptsize $100$};
  \node at ({52*\sc}, 0.5) {\scriptsize $110$};
  \node at ({6*\sc},  0.5) {\tiny $0001$};
  \node at ({18*\sc}, 0.5) {\tiny $0100$};
  \node at ({46*\sc}, 0.5) {\tiny $1011$};
  \node at ({58*\sc}, 0.5) {\tiny $1110$};

  %% Tick labels below
  \foreach \x/\l in {0/$0$, 8/{$\tfrac{1}{8}$}, 16/{$\tfrac{1}{4}$}, 24/{$\tfrac{3}{8}$}, 32/{$\tfrac{1}{2}$}, 40/{$\tfrac{5}{8}$}, 48/{$\tfrac{3}{4}$}, 56/{$\tfrac{7}{8}$}, 64/$1$} {
    \node[font=\tiny, below=-1pt] at ({\x*\sc}, 0) {\l};
  }

  %% ===== Combined annotations: orbit/fixed point + corresponding exceptional word =====
  \draw[->, thick] ({0.5*\sc},  1.9) -- ({0.5*\sc},  1.05);
  \node[font=\scriptsize, anchor=south, align=center] at ({0.5*\sc},  1.9)
    {\tiny exc.\ cell $C_{000000}$};

  \draw[->, thick] ({21.33*\sc}, 1.9) -- ({21.5*\sc}, 1.05);
  \node[font=\scriptsize, anchor=south, align=center] at ({21.33*\sc}, 1.9)
    {\tiny exc.\ cell $C_{010101}$};

  \draw[->, thick] ({42.67*\sc}, 1.9) -- ({42.5*\sc}, 1.05);
  \node[font=\scriptsize, anchor=south, align=center] at ({42.67*\sc}, 1.9)
    {\tiny exc.\ cell $C_{101010}$};

  \draw[->, thick] ({63.5*\sc}, 1.9) -- ({63.5*\sc}, 1.05);
  \node[font=\scriptsize, anchor=south, align=center] at ({63.5*\sc}, 1.9)
    {\tiny exc.\ cell $C_{111111}$};
\end{tikzpicture}
\caption{Partition for $P = 3$, $Q = 6$. $[0,1)$ is split into 20 intervals $C_w$. The four dark cells are the exceptional words $\{000000, 010101, 101010, 111111\}$, located at the orbits of the short-period points of the doubling map: the fixed points $0, 1$ and the period-2 orbit $\{\tfrac{1}{3}, \tfrac{2}{3}\}$. The partition becomes finer in their neighbourhoods.}
\label{fig:word_partition2}
\end{figure}
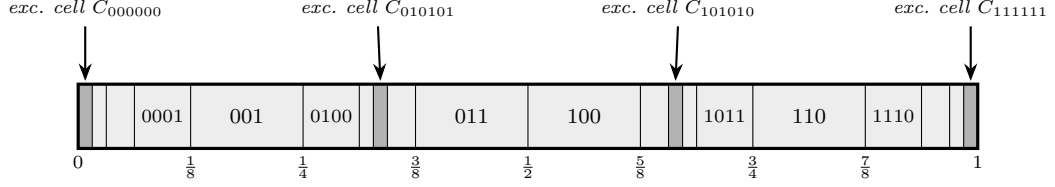